\documentclass[11.5pt,a4paper]{article}
\usepackage[top=1.25in, bottom=1.0in, left=1.0in, right=1.0in]{geometry}

\usepackage[T1]{fontenc}
\usepackage[utf8]{inputenc}
\usepackage{authblk}

\usepackage{geometry}
\usepackage{graphicx}
\usepackage{amssymb}
\usepackage{amsmath}
\numberwithin{equation}{section} 
\usepackage{amsthm}
\usepackage{diagbox}
\usepackage{empheq}
\usepackage{mdframed}
\usepackage{multirow}
\usepackage{mathrsfs}
\usepackage{booktabs}
\usepackage{lipsum}
\usepackage{graphicx}
\usepackage{color}
\usepackage{psfrag}
\usepackage{pgfplots}
\usepackage{bm}
\usepackage{float}
\usepackage{stmaryrd} 
\usepackage{subfigure}
\usepackage{appendix}

\usepackage[backref]{hyperref}

\newtheorem{alg}{Algorithm}[section]
\newtheorem{thm}{Theorem}[section]

\newtheorem{coro}{Corollary}[section]

\newtheorem{rmk}{Remark}[section]
\newtheorem{lem}{Lemma}[section]
\newtheorem{prop}{Proposition}[section]
\newtheorem{ex}{Example}[section]

\newcommand{\ssr}{
    \begin{tikzpicture}[baseline]
    \draw (0,-0.2em) -- (0em,0.8em) ;
    \draw (0.2em,-0.2em) -- (0.2em,0.8em);  
    \draw (0.4em,-0.2em) -- (0.4em,0.8em);  
    \end{tikzpicture}
    }

\newenvironment{keywords}
{\par\noindent\textbf{Keywords:}}

\graphicspath{{FIGS/}}

\definecolor{ocre}{RGB}{243,102,25}
\definecolor{mygray}{RGB}{243,243,244}
\definecolor{deepGreen}{RGB}{26,111,0}
\definecolor{shallowGreen}{RGB}{235,255,255}
\definecolor{deepBlue}{RGB}{61,124,222}
\definecolor{shallowBlue}{RGB}{235,249,255}

\newtheoremstyle{mytheoremstyle}{3pt}{3pt}{\normalfont}{0cm}{\rmfamily\bfseries}{}{1em}{{\color{black}\thmname{#1}~\thmnumber{#2}}\thmnote{\,--\,#3}}
\newtheoremstyle{myproblemstyle}{3pt}{3pt}{\normalfont}{0cm}{\rmfamily\bfseries}{}{1em}{{\color{black}\thmname{#1}~\thmnumber{#2}}\thmnote{\,--\,#3}}
\theoremstyle{mytheoremstyle}
\newmdtheoremenv[linewidth=1pt,backgroundcolor=shallowGreen,linecolor=deepGreen,leftmargin=0pt,innerleftmargin=20pt,innerrightmargin=20pt,]{theorem}{Theorem}[section]
\theoremstyle{mytheoremstyle}
\newmdtheoremenv[linewidth=1pt,backgroundcolor=shallowBlue,linecolor=deepBlue,leftmargin=0pt,innerleftmargin=20pt,innerrightmargin=20pt,]{definition}{Definition}[section]
\theoremstyle{myproblemstyle}
\newmdtheoremenv[linecolor=black,leftmargin=0pt,innerleftmargin=10pt,innerrightmargin=10pt,]{problem}{Problem}[section]

\usepgfplotslibrary{colorbrewer}
\pgfplotsset{width=8cm,compat=1.9}

\title{\LARGE Linearly implicit energy-preserving integrating factor methods for the 2D nonlinear Schr\"odinger equation with wave operator and convergence analysis }
\author[a]{Xuelong Gu}
\author[a]{Wenjun Cai}
\author[b]{Chaolong Jiang}
\author[a]{Yushun Wang \thanks{Corresponding author: wangyushun@njnu.edu.cn}}
\affil[a]{Jiangsu Key Laboratory for NSLSCS, School of Mathematical Sciences, Nanjing Normal University, Nanjing, 210023, China. }
\affil[b]{School of Statistics and Mathematics, Yunnan University of Finance and Economics, Kunming, 650221, China. \vspace{-4em} }
\date{}
\begin{document}
\maketitle
{\noindent}	 \rule[-10pt]{15.5cm}{0.1em}
\begin{abstract}
			 In this paper, we develop a novel class of linear energy-preserving integrating factor methods for the 2D nonlinear Schr\"odinger equation with wave operator (NLSW), combining the scalar auxiliary variable approach and the integrating factor methods. A second-order scheme is proposed, which is rigorously proved to be energy-preserving. By using the energy methods, we analyze its optimal convergence in the $H^1$ norm without any restrictions on the grid ratio, where a novel technique and an improved induction argument are proposed to overcome the difficulty posed by the unavailability of a priori $L^\infty$ estimates of numerical solutions. Based on the integrating factor Runge-Kutta methods, we extend the proposed scheme to arbitrarily high order, which is also linear and conservative. Numerical experiments are presented to confirm the theoretical analysis and demonstrate the advantages of the proposed methods.
\end{abstract}
\begin{small}
\end{small}
{\noindent}	 \rule[-10pt]{15.5cm}{0.1em}
\section{Introduction}
In this paper, we consider the following 2D nonlinear Schr\"odinger equation with wave operator (NLSW)
\begin{equation}\label{eq1}
\begin{aligned}
					 &\partial_{tt} u(\bm{x}, t) + i \alpha \partial_t u(\bm{x}, t) - \Delta u(\bm{x}, t) + \beta|u(\bm{x}, t)|^2u(\bm{x}, t) = 0, \quad (\bm{x}, t) \in \Omega \times (0, T], \\ 
				   &u(\bm{x}, 0) = u_0(\bm{x}), \ u_t(\bm{x}, 0) = u_1(\bm{x}), \quad \bm{x} \in \Omega, \\ 
					 &u(\bm{x}, t) = 0, \quad (\bm{x}, t) \in \partial \Omega \times (0, T], 
\end{aligned}
\end{equation}
where $u(\bm{x}, t)$ is a complex function, $\Omega = (x_L, x_L + X) \times (y_L, y_L + Y) \subset \mathbb{R}^2$ is a bounded open set, $\alpha \neq 0$ and $\beta \neq 0$ are real constants, $i = \sqrt{-1}$ is the complex unit and $\Delta$ is the Laplacian operator. The above NLSW arises from different applications of physics, such as the Langmuir wave envelope approximation in plasma \cite{001}, the nonrelativistic limit of the Klein-Gordon equation \cite{002} and the modulated planar approximation of the sine-Gordon equation for light bullets \cite{003}. It is worth noting that the initial-boundary value problem \eqref{eq1} preserves the energy: 
\begin{equation}\label{eq1-energy}
	E(t) = \int_\Omega \Big( |\nabla u(\bm{x}, t)|^2 + |\partial_t u(\bm{x}, t)|^2 + \frac{\beta}{2} |u(\bm{x}, t)|^4 \Big) d\bm{x} = E(0).
\end{equation}

Numerous theoretical and practical discoveries reveal that algorithms that can preserve a discrete counterpart of conservative laws often have good numerical behaviours. As a result, constructing a numerical approach to satisfy energy conservation law \eqref{eq1-energy} at a discrete level for \eqref{eq1} will be intriguing. The classical energy-preserving algorithms include the discrete gradient methods \cite{dG1,dG2}, the averaged vector field methods \cite{avf1, avf2} and the Hamiltonian boundary value methods \cite{HBVM2,HBVM1}, etc. Specifically, for the NLSW, Zhang et al. constructed energy-preserving methods \cite{dG_NLSW} based on the discrete gradient method. Brugnano et al. applied the Hamiltonian boundary value methods to the NLSW in \cite{HBVM_NLSW}. A local energy-preserving method was introduced in \cite{LSP_NLSW} by Huang et al. However, the aforementioned numerical schemes are fully implicit, and a nonlinear system must be solved by some iterative methods which makes them time-consuming. To improve the efficiency, linear energy-preserving schemes based on the leap-frog method were devised for the NLSW in \cite{NLSW_PDG,NLSW_PDG_FE,Wang_linear,NLSW_Li,NLSW_Pdg_JCAM,IEQ_linear}. More recently, the scalar auxiliary variable (SAV) approach proposed by Shen in \cite{SAV-Shen,SAV-Shen-SIAM} has been proved to be a particularly effective tool to construct linear schemes. Although the SAV approach was first developed to simulate gradient flow systems, it has also been successfully extended to conservative systems in terms of developing linear energy-preserving methods \cite{SAV-Li, SAV-NLSW}.

To reach high precision and stability while simulating extremely stiff differential equations, such as highly oscillatory ODEs and semi-discrete time-dependent PDEs, exponential integrators that involve exact integration of the linear part of the target systems are preferable. Readers are referred to the remarkable review work by Hochbruck and Ostermann \cite{exponential-intergrators} for details. Recently, numerous energy-preserving exponential integrators for conservative systems have been developed. Li and Wu constructed a second-order energy-preserving exponential AVF (EAVF) approach in \cite{Wu16SIAM}. However, the proposed scheme is fully implicit. Although some explicit exponential integrators were proposed in \cite{ex_siam,ex_jamc,ex_jcam}, but they failed to be energy-preserving. Gu et al. developed linearly implicit exponential partitioned AVF methods in \cite{EPAVF}. Nevertheless, it should be applied to so-called multi-components Hamiltonian systems to obtain linear schemes. In addition, Jiang et al. in \cite{ESAV-KG} constructed linear energy-preserving exponential integrators by combining the exponential time difference (ETD) methods and the SAV approach. The above exponential integrators only have second-order accuracy. Higher-order energy-preserving exponential integrators have also been developed extensively. Mei et al. extended the second-order EAVF method to arbitrarily high-order by using the modified vector-field technique in \cite{EAVF-high}. In \cite{wu-2022} and \cite{wang-nls-etd}, the authors developed arbitrarily high-order continuous-stage energy-preserving ETD methods. A class of linear high-order conservative exponential integrators for the nonlinear Schr\"odinger equations (NLSE) was introduced in \cite{ESAV-High}. Taking the NLSW as an example, we also propose a novel class of linear and conservative exponential integrators, which combines the SAV approach and the integrating factor methods. Compared with the method proposed in \cite{ESAV-KG}, our method is not only more convenient to be extended to arbitrarily high-order but also amenable to perform convergence analysis.

Although many of the above energy-preserving exponential integrators have been developed, most of them concentrated on their construction and implementation \cite{cui-mass-energy, ESAV-High}, and few references are concerned with their convergence analysis, especially for the 2D problems. For the 1D problems, the convergence result relies heavily on the discrete version of the following 1D Sobolev inequality
\begin{equation}\label{Sobolev1d}
				\|u\|_{L^\infty} \leq C\|u\|_{H^1}, \ \forall u \in H_0^1(\Omega) \ \text{with} \ \Omega \subset \mathbb{R}.
\end{equation}
A priori $L^\infty$ estimates for the numerical solutions can usually be derived from the conservation laws, see \cite{NLSW_PDG, Wang_linear, NLSW_Li}. However, this idea cannot be extended to 2D as \eqref{Sobolev1d} is no longer valid, which makes the error estimates for 2D problems more difficult. Although there were some works on the convergence of exponential integrators for the Allen-Cahn type equations \cite{du_2019,du_2021,ju_ifrk_high,zhang_anm}, these analyses rely on the so-called maximal bounded principle and cannot be extended to other equations. In \cite{etd-ch}, Li et al. performed an induction argument to establish convergence results of the ETD methods with a strict restriction on the grid ratio. Apart from these, no analyses of integrating factor methods were carried out for general 2D nonlinear problems to the best of our knowledges. Taking the 2D NLSW as an example, we also perform unconditional convergence analysis of the proposed second-order scheme (named \textbf{SAV-IF} method) for both $\beta > 0$ and $\beta < 0$. Inspired by \cite{Wang-NLS}, a new technique that requires only the $H^1$ a priori estimate of the numerical solution is employed to get the unconditional convergence result for the NLSW when $\beta > 0$. For $\beta < 0$, such a technique  cannot be applied straightforwardly as the $H^1$ boundedness is no longer available from the discrete energy conservation law here. Although Wang et al. introduced the ``lifting'' technique in \cite{lifting-wang} to obtain the unconditional convergence result without any a priori estimate for the numerical solution, such a technique is failed for the convergence analysis of the exponential integrators. Therefore, we provide an improved induction argument to obtain unconditional convergence results for $\beta < 0$, which has no restrictions on the grid ratio.

The rest of this paper is organized as follows: In Section \ref{sec.2}, we recast the NLSW \eqref{eq1} into an equivalent one by using the SAV approach, followed by proposing a fully discrete \textbf{SAV-IF} scheme. We not only prove its energy conservation, but also display that it can be implemented efficiently. In Section \ref{sec.3}, we establish the unconditionally optimal $H^1$ error estimates for the \textbf{SAV-IF} method for both $\beta > 0$ and $\beta < 0$. We extend arbitrarily high-order and linear energy-preserving methods in Section \ref{sec.4}. Numerical examples are performed to confirm the theoretical results and demonstrate the superiority behavior of our methods over the existing energy-preserving algorithms in Section \ref{sec.5}. Some conclusions are covered in the last section.
 
\section{Numerical scheme}\label{sec.2}

\subsection{SAV reformulation}
In this section, we introduce the SAV reformulation for the NLSW equation. The reformulated system preserves a quadratic energy and is equivalent to the original one under consistent initial conditions. The SAV reformulation will provide an elegant platform to develop linear energy-preserving exponential integrators.

Let $v(\bm{x}, t) = \partial_t u(\bm{x}, t)$ and introduce an auxiliary variable $r(t)$ such that
\begin{equation*}
				 r(t) = \sqrt{\int_\Omega G(u(\bm{x}, t)) d\bm{x} + C_0} := \sqrt{H[u(\bm{x}, t)]} \quad \text{with} \quad G(u(\bm{x}, t)) = \frac{1}{2} |u(\bm{x}, t)|^4 ,
\end{equation*}
where $C_0 > 0$ is a constant to guarantee $H[u(\bm{x}, t)] > 0$. Denote $g(u(\bm{x}, t)) = |u(\bm{x}, t)|^2 u(\bm{x}, t)$, $f(u(\bm{x}, t)) = \frac{g(u(\bm{x} ,t))}{\sqrt{H[u(\bm{x}, t)]}}$. System \eqref{eq1} is then rewritten into an equivalent one according to the SAV approach \cite{ESAV-High,ESAV-KG,SAV-Li,SAV-NLSW} as
\begin{equation}\label{sav}
\left\lbrace
\begin{aligned}
				&\partial_t u(\bm{x}, t) = v(\bm{x}, t), \\
				&\partial_t v(\bm{x}, t) +   i \alpha v(\bm{x}, t) - \Delta u(\bm{x}, t) + \beta r(t) f(u(\bm{x}, t))  = 0, \\ 
				&\frac{dr(t)}{dt} =  \Re (f(u(\bm{x}, t)), v(\bm{x}, t)),
\end{aligned}
\right.
\end{equation}
where $\Re$ represents taking the real part of a complex function and $(u, v) = \int u\overline{v} d\bm{x}$ is the $L^2$ inner product.
By taking the $L^2$ inner products on both sides of the equations in \eqref{sav}, respectively with $\partial_t v$, $\partial_t u$ and $r$, it is readily to show that the solution of system \eqref{sav} preserves the following quadratic energy
\begin{equation}\label{sav-energy}
				\dfrac{dE(t)}{dt} = 0, \ E(t) = \int_\Omega |\nabla u(\bm{x}, t)|^2 + |v(\bm{x}, t)|^2 d\bm{x} + \beta r^2(t) - \beta C_0.
\end{equation}
We emphasize that the SAV reformulation \eqref{sav} is equivalent to the original one and their energy conservation laws \eqref{eq1-energy}, \eqref{sav-energy} are the same if consistent initial conditions $r(0) = \sqrt{H[u_0(\bm{x})]}$ and $v(\bm{x}, 0) = u_1(\bm{x})$ are imposed. In the following sections, we develop linear energy-preserving exponential integrators for \eqref{sav}, which in turn solve the original system.
\begin{rmk}
				We note that the last equation of \eqref{sav} is obtained by combining the first equation of \eqref{sav} and the identity
				\begin{equation}\label{rt-2}
								\frac{dr(t)}{dt} = \Re (f(u(\bm{x}, t)), \partial_t u(\bm{x}, t)).
				\end{equation}
				In the early works related to the SAV approach, the governing system was usually extended by \eqref{rt-2} \cite{SAV-Li,SAV-NLSW,SAV-Shen}, then the resulting reformulation was amenable to simple and efficient conservative numerical schemes. For the construction of the integrating factor methods, expanding the original system like \eqref{sav} is not only important for developing energy preserving methods but also convenient for the theoretical analysis of the resulting schemes, which will be demonstrated later.
\end{rmk}
\subsection{Spatial discretization}

Given a positive integer $N$, we introduce the spatial mesh sizes $h_1 = \frac{X}{N}$, $h_2 = \frac{Y}{N}$ and the following index set 
\begin{equation*}
				\mathcal{T}_\mathcal{N} = \{(j,k)| j = 1,2,\cdots,N-1, \quad k = 1,2,\cdots,N-1\}.
\end{equation*}
 Let $\Omega_\mathcal{N} = \{(x_j, y_k)| x_j = x_L + jh_1, y_k = y_L + kh_2, (j,k) \in \mathcal{T}_\mathcal{N}\}$ be the spatial grid points. All of the 2D complex-valued grid functions with zero boundary values defined on the $\Omega_\mathcal{N}$ are denoted $\mathcal{M}_\mathcal{N}$. For any functions $u, v \in \mathcal{M}_\mathcal{N}$, we define the discrete $L^2$ inner product as
\begin{equation*}
				(u, v)_{l^2} = h_1 h_2 \sum\limits_{j=1}^{N-1}\sum\limits_{k=1}^{N-1} u_{jk}\overline{v}_{jk}.
\end{equation*}
The discrete $L^2, L^p, L^\infty$ norms are
\begin{equation*}
\|u\|_{l^2} = (u, u)_{l^2}^{\frac{1}{2}}, \quad \|u\|_{l^p} = \Big(h_1h_2 \sum\limits_{j=1}^{N-1}\sum\limits_{k=1}^{N-1}|u_{jk}|^p\Big)^{\frac{1}{p}}, \quad \|u\|_{l^\infty} = \max\limits_{ 0 \leq j \leq N-1 \atop 0 \leq k \leq N-1 } |u_{jk}|.
\end{equation*}

Since the boundary conditions are predetermined to be homogeneous, we can employ the sine pseudo-spectral method for spatial discretization to guarantee accuracy and efficiency. Given a function $u \in \mathcal{M}_\mathcal{N}$, its 2D discrete sine transform $\widehat{u} = \mathcal{S} u$ is defined as
\begin{equation}\label{interp-u}
				\widehat{u}_{pq} = \frac{4}{N^2}\sum\limits_{j=1}^{N-1} \sum\limits_{k=1}^{N-1} u_{jk} \sin{\left( \mu_p(x_j - x_L) \right)} \sin{\left(\nu_q (y_k - y_L)\right)}, \quad (p, q)  \in \mathcal{T}_\mathcal{N},
\end{equation}
where $\mu_p = \frac{p\pi}{X}, \ \nu_q = \frac{q\pi}{Y}$. Due to the orthogonality of the sine basis, we can reconstruct the function $u$ by the inverse transform $u = \mathcal{S}^{-1} \widehat{u}$ as
\begin{equation}
				u_{jk} = \sum\limits_{p = 1}^{N-1}\sum\limits_{q=1}^{N-1} \widehat{u}_{pq} \sin{\left( \mu_p (x_j - x_L)  \right)} \sin{\left(\nu_q (y_j - y_L)\right)}, \quad (j, k) \in \mathcal{T}_\mathcal{N}.
\end{equation}
Denote by $\widehat{\mathcal{M}}_\mathcal{N} = \{\mathcal{S} u | u \in \mathcal{M}_\mathcal{N}\}$. To obtain an approximation of the Laplacian operator, we introduce two operators $\widehat{D}_{xx}, \ \widehat{D}_{yy}$ on $\widehat{\mathcal{M}}_\mathcal{N}$, such that
\begin{equation*}
				(\widehat{D}_{xx} \widehat{u})_{pq} = -\mu_p^2 \widehat{u}_{pq}, \quad (\widehat{D}_{yy}\widehat{u})_{pq} = -\nu_q^2 \widehat{u}_{pq}.
\end{equation*}
Consequently, the sine pseudo-spectral approximation to the second-order differential operators $\partial_{xx}$ and $\partial_{yy}$ can be defined, respectively as
\begin{equation*}
				D_{xx} = \mathcal{S}^{-1} \widehat{D}_{xx} \mathcal{S}, \quad D_{yy} = \mathcal{S}^{-1} \widehat{D}_{yy} \mathcal{S}.
\end{equation*}
Then, we can propose the approximation of the Laplacian operator as follows
\begin{equation*}
				\Delta_\mathcal{N} u = (D_{xx} + D_{yy})u, \quad \forall u \in \mathcal{M}_\mathcal{N},
\end{equation*}
and the corresponding operator on the $\widehat{\mathcal{M}}_\mathcal{N}$ is denoted $\widehat{\Delta}_\mathcal{N}$, such that
\begin{equation*}
				(\widehat{\Delta}_\mathcal{N} \widehat{u})_{pq} = -(\mu_p^2 + \nu_q^2) \widehat{u}_{pq}, \quad (p,q)\in \mathcal{T}_\mathcal{N}
\end{equation*}
In the rest of this article, we will denote by $\lambda_{pq}^2 = \mu_p^2 + \nu_q^2$. It is noteworthy that the actions of $\mathcal{S}$ and $\mathcal{S}^{-1}$ can be implemented by available routines such as \texttt{dst.m} and \texttt{idst.m} in \texttt{Matlab} with the computational cost $\mathcal{O}(N^2 \log{N})$ \cite{shen2011spectral}.
\begin{rmk}
				In previous works, the discrete Laplacian $\Delta_\mathcal{N}$ is usually treated in terms of spectral differential matrices \cite{Gong-NLS,GP-SPS}. Since the construction of exponential integrators requires the powers of $\Delta_\mathcal{N}$, we prefer to treat $\Delta_\mathcal{N}$ as a linear operator from now on. Then, the powers of $\Delta_\mathcal{N}$ can be characterized in conjunction with the 2D discrete sine transform and its eigenvalues. A similar idea can be found in \cite{Ju-MBE-ETD,Qiao-MBE-Splitting}.
\end{rmk}
It is useful to define the following discrete semi $H^1$ and semi $H^2$ norms with respect to the sine pseudo-spectral method
\begin{equation*}
				|u|_{\mathcal{N},1} = \sqrt{(-\Delta_\mathcal{N} u, u)_{l^2}}, \ |u|_{\mathcal{N},2} = \sqrt{(\Delta_\mathcal{N} u, \Delta_\mathcal{N} u)_{l^2}}.
\end{equation*}

With the above preparations, the sine pseudo-spectral method discretization for \eqref{sav} is to find $(u_\mathcal{N}(t), v_\mathcal{N}(t), r_\mathcal{N}(t)) \in \mathcal{M}_\mathcal{N} \times \mathcal{M}_\mathcal{N} \times \mathbb{R}$, such that
\begin{equation}\label{sav-semi}
\left\lbrace
\begin{aligned}
				\dot{u}_\mathcal{N}(t) &= v_\mathcal{N}(t), \\ 
				\dot{v}_\mathcal{N}(t) &= \Delta_\mathcal{N} u_\mathcal{N}(t) - i\alpha v_\mathcal{N}(t) - \beta r_\mathcal{N}(t) f_\mathcal{N}(u_\mathcal{N}(t)), \\ 
				\dot{r}_\mathcal{N}(t) &= \Re (f_\mathcal{N}(u_\mathcal{N}(t)), v_\mathcal{N}(t))_{l^2},
\end{aligned}
\right.
\end{equation}
where ``$\cdot$'' denotes the derivative with respect to $t$, $\{u_{\mathcal{N}}(t)\}_{jk} = u_{jk}(t)$. It is notable that the inner product in \eqref{sav} is replaced by the discrete one and $f_\mathcal{N}$ is the discrete version of $f$, i.e.,
\begin{equation*}
				 f_\mathcal{N} (u_\mathcal{N}(t)) = \frac{g(u_\mathcal{N}(t))}{\sqrt{H_\mathcal{N}(u_\mathcal{N}(t))}}, \quad H_\mathcal{N}(u_\mathcal{N}(t)) = (G(u_\mathcal{N}(t)), 1)_{l^2} + C_0,
\end{equation*}
where $\{G(u_\mathcal{N}(t))\}_{jk} = \frac{1}{2}|u_{jk}(t)|^4$ and $\{g(u_\mathcal{N}(t))\}_{jk} = |u_{jk}(t)|^2u_{jk}(t)$ for $(j,k) \in \mathcal{T}_\mathcal{N}$. 
\begin{thm}
				System \eqref{sav-semi} possesses the following semi-discrete energy conservation law
				\begin{equation*}\label{discrete-energy}
								\dfrac{d E_\mathcal{N}(t)}{dt} = 0, \ E_\mathcal{N}(t) = |u_\mathcal{N}(t)|^2_{\mathcal{N},1} + \|v_\mathcal{N}(t)\|^2_{l^2} + \beta r^2_\mathcal{N}(t) - \beta C_0,
				\end{equation*}
\end{thm}
\begin{proof}
				By taking the discrete inner products on both sides of \eqref{sav-semi} with $\dot{v}_\mathcal{N}(t)$, $\dot{u}_\mathcal{N}(t)$ and $2r_\mathcal{N}(t)$, respectively, the result of the conservation law can be obtained straightforwardly.
\end{proof}
\subsection{A second-order energy-preserving SAV-IF time integrator}\label{integrator}
Given a positive integer $M$, the time domain is partitioned uniformly with a step size $\tau = \frac{T}{M}$. We denote by $\Omega^\tau = \{ t_n = n \tau | n = 0, 1, \cdots M \}$. Given a time grid function $u^n$, we define
\begin{equation*}
				\delta_t u^{n+\frac{1}{2}} = \frac{u^{n+1} - u^n}{\tau}, \quad u^{n+\frac{1}{2}} = \frac{u^n + u^{n+1}}{2}, \quad \widetilde{u}^{n+\frac{1}{2}} = \frac{3 u^n - u^{n-1}}{2}.
\end{equation*}
By setting $\bm{z}_\mathcal{N}(t) = (u_\mathcal{N}(t), v_\mathcal{N}(t)) \in \mathcal{M}_\mathcal{N} \times \mathcal{M}_\mathcal{N}$, and
\begin{equation}\label{matA}
				A_\mathcal{N} = \begin{pmatrix} 0 & I \\ \Delta_\mathcal{N} & -i\alpha I \end{pmatrix}, \ F_\mathcal{N}(\bm{z}_\mathcal{N}(t)) = \begin{pmatrix} 0 \\ \beta r_\mathcal{N}(t) f_\mathcal{N}(u_\mathcal{N}(t)) \end{pmatrix},
\end{equation}
where $I$ represents the identity operator on $\mathcal{M}_\mathcal{N}$. The first two equations of \eqref{sav-semi} can be recast into a more compact form as 
\begin{equation}\label{origin-compact}
				\dot{\bm{z}}_\mathcal{N} = A_\mathcal{N} \bm{z}_\mathcal{N}(t) - F_\mathcal{N}(\bm{z}_\mathcal{N}(t)).
\end{equation}

In construction of the integrating factor methods, we introduce the Lawson transform \cite{Lawson} $\bm{\psi}_\mathcal{N}(t) = (\phi_\mathcal{N}(t), \varphi_\mathcal{N}(t)) = e^{-t A_\mathcal{N}} \bm{z}_\mathcal{N}(t)$ as well as its inverse $\bm{z}_\mathcal{N}(t) = e^{t A_\mathcal{N}} \bm{\psi}_\mathcal{N}(t)$. Here, $e^{t A_\mathcal{N}} = \sum_{j=0}^\infty \frac{(t A_\mathcal{N})^j}{j!}$. To be more preciously, we denote by
\begin{equation*}
				e^{t A_\mathcal{N}} = 
				\begin{pmatrix} 
								e^{11}(t) & e^{12}(t) \\ 
								e^{21}(t) & e^{22}(t)
				\end{pmatrix}.
\end{equation*}
Then the elements $e^{\mu\nu}(t)$, $(\mu, \nu = 1,2)$ are characterized as below.
\begin{prop}\label{exp-element2d}
				For any $\widehat{u} \in \widehat{\mathcal{M}}_\mathcal{N}$, let $\widehat{e}^{\mu\nu}(t) = \mathcal{S} e^{\mu\nu}(t) \mathcal{S}^{-1} \ (\mu, \nu = 1,2)$. The action of $\widehat{e}^{\mu\nu}(t)$ on $\widehat{u}$ can be implemented via 
				\begin{equation*}
								(\widehat{e}^{\mu\nu}(t) \widehat{u})_{pq} = \widehat{e}^{\mu\nu}_{pq}(t) \widehat{u}_{pq}, \quad (p,q) \in \mathcal{T}_\mathcal{N},
				\end{equation*}
				where the eigenvalues $\widehat{e}^{\mu\nu}_{pq}$ are
				\begin{equation*}
				\begin{aligned}
								&\widehat{e}^{11}_{pq} (t) = \frac{\omega_{pq}^+ e^{i\omega_{pq}^-t} - \omega_{pq}^- e^{i\omega_{pq}^+t}}{\omega_{pq}}, \quad \widehat{e}^{12}_{pq}(t) = \frac{e^{i\omega_{pq}^+t}- e^{i\omega_{pq}^-t}}{i\omega_{pq}},  \\
								&\widehat{e}^{21}_{pq}(t) = -\lambda_{pq}^2 \frac{e^{i\omega_{pq}^+t}- e^{i\omega_{pq}^-t}}{i\omega_{pq}}, \quad \widehat{e}^{22}_{pq}(t) = \frac{\omega_{pq}^+ e^{i\omega_{pq}^+t} - \omega_{pq}^- e^{i\omega_{pq}^-t}}{\omega_{pq}},
				\end{aligned}
				\end{equation*}
				with $\omega_{pq}^{\pm} = -\frac{\alpha \pm \sqrt{\alpha^2 + 4\lambda_{pq}^2}}{2}$ and $\omega_{pq} = \omega_{pq}^+ - \omega_{pq}^-$.
\end{prop}
Then, system \eqref{origin-compact} and the third equation of \eqref{sav-semi} are presented in terms of new variable $\bm{\psi}_\mathcal{N}(t) = (\phi_\mathcal{N}(t), \varphi_\mathcal{N}(t))$ as follows:
\begin{equation}\label{Lawson-system}
\left\lbrace
\begin{aligned}
				\dot{\bm{\psi}}_\mathcal{N}(t) &= -e^{-t A_\mathcal{N}} F_\mathcal{N}(e^{t A_\mathcal{N}}\bm{\psi}_\mathcal{N}(t)), \\
				\dot{r}_\mathcal{N}(t) &= \Re ( f_\mathcal{N}(e^{11}(t) \phi_\mathcal{N}(t) + e^{12}(t) \varphi_\mathcal{N}(t)), e^{21}(t)\phi_\mathcal{N}(t) + e^{22}(t) \varphi_\mathcal{N}(t) )_{l^2}.
\end{aligned}
\right.
\end{equation}
Let $n \geq 2$, we discretize \eqref{Lawson-system} by the midpoint rule and the extrapolation technique to get
\begin{equation}\label{scheme-lawson1}
\left\lbrace
\begin{aligned}
				\delta_t \bm{\psi}^{n+\frac{1}{2}} &= -e^{-(t_n + \frac{\tau}{2})A_\mathcal{N}} F_\mathcal{N}\big(\tfrac{3}{2} e^{t_n A_\mathcal{N}} \bm{\psi}^n - \tfrac{1}{2} e^{t_{n-1} A_\mathcal{N}} \bm{\psi}^{n-1}\big), \\
				\delta_t r^{n+\frac{1}{2}} &=  \Re \big(
				f_\mathcal{N}\big(\widetilde{(e^{11}\phi)}^{n+\tfrac{1}{2}} + \widetilde{(e^{12}\varphi)}^{n+\tfrac{1}{2}}\big) , (e^{21}\phi)^{n+\tfrac{1}{2}} + (e^{22} \varphi)^{n+\tfrac{1}{2}} \big)_{l^2},
\end{aligned}
\right.
\end{equation}
where $\widetilde{(e^{11}\phi)}^{n+\frac{1}{2}} = \tfrac{3}{2}e^{11}(t_n)\phi^n- \frac{1}{2} e^{11}(t_{n-1})\phi^{n-1}$. 

In the practical implementation, it is preferable to provide a discretizaion in terms of the original variables. To this end, we employ the discrete Lawson transform $\bm{\psi}^n = e^{-t_n A_\mathcal{N}} \bm{z}^n$ as well as its inverse $\bm{z}^n = e^{t_n A_\mathcal{N}}\bm{\psi}^n$ to \eqref{scheme-lawson1}, then  perform a componentwise expression of the resulting system to get the  following \textbf{SAV-IF} method.
\begin{alg}[\textbf{SAV-IF} method]
\begin{equation}\label{fully1}
\left\lbrace 
\begin{aligned}
				u^{n+1} &= e^{11}(\tau) u^n + e^{12}(\tau) v^n - \tau \beta r^{n+\frac{1}{2}} e^{12}\big(\tfrac{\tau}{2}\big) f_\mathcal{N}(\widetilde{u}^{n+\frac{1}{2}}), \\ 
				v^{n+1} &= e^{21}(\tau) u^n + e^{22}(\tau) v^n - \tau \beta r^{n+\frac{1}{2}} e^{22}\big(\tfrac{\tau}{2}\big) f_\mathcal{N}(\widetilde{u}^{n+\frac{1}{2}}), \\
				\delta_t r^{n+\frac{1}{2}} &= \Re ( f_\mathcal{N}(\widetilde{u}^{n+\frac{1}{2}}), \mathcal{A}(u^n, u^{n+1}, v^n, v^{n+1}) )_{l^2} ,
\end{aligned}
\right.
\end{equation}
where the operator $\mathcal{A}$ is defined as follows:
\begin{equation*}
\begin{aligned}
				\mathcal{A}(u^n, u^{n+1}, u^n, u^{n+1}) &=  \tfrac{1}{2}\left(e^{21}\big(\tfrac{\tau}{2}\big)u^n + e^{21}\big(-\tfrac{\tau}{2}\big)u^{n+1}+ e^{22}\big(\tfrac{\tau}{2}\big)v^n + e^{22}\big(-\tfrac{\tau}{2}\big)v^{n+1}\right).
\end{aligned}
\end{equation*}
Since \eqref{fully1} is a three level scheme, we let $\widetilde{u}^{\frac{1}{2}} = u^0$ for $n = 0$ instead of the extrapolation. Although the approach to get $u^1$ is only of first-order accuracy, it will not affect the overall convergence rate since we only use it once. 
\end{alg}
\begin{lem}\label{identity}
				 Let $L_\mathcal{N} = \left(\begin{smallmatrix} -\Delta_\mathcal{N} & 0 \\ 0 & I_\mathcal{N} \end{smallmatrix}\right)$ and $\bm{z}^n \in \mathcal{M}_\mathcal{N} \times \mathcal{M}_\mathcal{N}$. The following identity holds.
				 \begin{equation*}
								\left\langle e^{\tau A_\mathcal{N}} \bm{z}^n, A_\mathcal{N} e^{\tau L_\mathcal{N}} \bm{z}^n \right\rangle_{l^2} = \left\langle \bm{z}^n, A_\mathcal{N} \bm{z}^n \right\rangle_{l^2},
				 \end{equation*}
where the discrete $l^2$ inner product between the vector-valued functions is defined by $\left\langle \bm{z}_1, \bm{z}_2 \right\rangle_{l^2} = (u_1, v_1)_{l^2} + (u_2, v_2)_{l^2}, \ \forall \bm{z}_\nu = (u_\nu, v_\nu) \in \mathcal{M}_\mathcal{N} \times \mathcal{M}_\mathcal{N}, \ \nu = 1,2$.
\end{lem}
\begin{proof}
				Notice that $A_\mathcal{N} = J_\mathcal{N} L_\mathcal{N}$ in \eqref{matA}, with $J_\mathcal{N} = \left(\begin{smallmatrix} 0 & I_\mathcal{N} \\ -I_\mathcal{N} & -i\alpha I_\mathcal{N} \end{smallmatrix}\right)$. It is readily to verify that $J_\mathcal{N}$ is skew-adjoint i.e.,
				\begin{equation*}
								\left\langle \bm{z}_1, J_\mathcal{N} \bm{z}_2 \right\rangle_{l^2} = - \left\langle J_\mathcal{N} \bm{z}_1,  \bm{z}_2 \right\rangle_{l^2},
				\end{equation*}
				which implies $\Re \left\langle J_\mathcal{N} \bm{z}, \bm{z} \right\rangle_{l^2} = 0.$ Then, we consider the following linear initial-value problem
				\begin{equation}\label{ivp}
								\dot{\bm{z}} = J_\mathcal{N}L_\mathcal{N} \bm{z}, \ \bm{z}(0) = \bm{z}^0.
				\end{equation}
				Taking the inner product on both sides of \eqref{ivp}, and getting the real part of the resulting equation, we have
				\begin{equation*}
								\dfrac{d}{dt} \left\langle \bm{z}, L_\mathcal{N} \bm{z} \right\rangle_{l^2} = 0,
				\end{equation*}
				which implies that \eqref{ivp} is a conservative system with the first integral $\left\langle \bm{z}, L_\mathcal{N} \bm{z} \right\rangle_{l^2}$. Since the exact solution of \eqref{ivp} is $\bm{z}(t) = e^{t A_\mathcal{N}}\bm{z}^0$. The result of Lemma \ref{identity} is straightforward.
\end{proof}
\begin{thm}\label{fully-conservation}
				Scheme \eqref{fully1} satisfies the following fully discrete energy conservation law
				\begin{equation}\label{conservation fully 1}
								E_\mathcal{N}^{M} = E_\mathcal{N}^{M-1} = \cdots = E_\mathcal{N}^0, \quad E_\mathcal{N}^n = |u^n|_{\mathcal{N},1}^2 + \|v^n\|_{l^2}^2 + \beta (r^n)^2 - \beta C_0.
				\end{equation}
\end{thm}
\begin{proof}
				We recast the first two equations of \eqref{fully1} into the following compact form
				\begin{equation}\label{savif-compact}
								\frac{e^{-\frac{\tau A_\mathcal{N}}{2}} \bm{z}^{n+1} - e^{\frac{\tau A_\mathcal{N}}{2}} \bm{z}^n}{\tau} = - F_\mathcal{N}(\widetilde{\bm{z}}^{n+\frac{1}{2}}).
				\end{equation}
				Taking the discrete inner product on both sides of \eqref{savif-compact} with $\frac{1}{2}L_\mathcal{N}(e^{-\frac{\tau A_\mathcal{N}}{2}}\bm{z}^{n+1} + e^{\frac{\tau A_\mathcal{N}}{2}} \bm{z}^n)$, we obtain
				\begin{equation}\label{conservation fully 2}
				\begin{aligned}
								&\frac{1}{\tau}\langle e^{-\frac{\tau A_\mathcal{N}}{2}} \bm{z}^{n+1} - e^{\frac{\tau A_\mathcal{N}}{2}} \bm{z}^n, L_\mathcal{N}( e^{-\frac{\tau A_\mathcal{N}}{2}} \bm{z}^{n+1} + e^{\frac{\tau A_\mathcal{N}}{2}} \bm{z}^n) \rangle_{l^2} \\
								&= -\langle F_\mathcal{N}(\widetilde{\bm{z}}^{n+\frac{1}{2}}), L_\mathcal{N}( e^{-\frac{\tau A_\mathcal{N}}{2}}\bm{z}^{n+1} + e^{\frac{\tau A_\mathcal{N}}{2}} \bm{z}^n) \rangle_{l^2},
				\end{aligned}
				\end{equation}
				which can be further simplified as follows:
				\begin{equation}\label{conservation fully 3}
				\begin{aligned}
								&\frac{1}{\tau}\big( \langle \bm{z}^{n+1}, L_\mathcal{N} \bm{z}^{n+1} \rangle_{l^2} - 
								\langle \bm{z}^n, L_\mathcal{N} \bm{z}^n \rangle_{l^2} + 2i  \Im \langle e^{\frac{\tau A_\mathcal{N}}{2}} \bm{z}^{n+1}, L_\mathcal{N} e^{\frac{\tau A_\mathcal{N}}{2}} \bm{z}^n \rangle_{l^2}   \big), \\ 
								&= -2r^{n+\frac{1}{2}}  ( f_\mathcal{N}(\widetilde{u}^{n+\frac{1}{2}}), \mathcal{A}(u^n, u^{n+1}, v^n, v^{n+1}) )_{l^2}, 
				\end{aligned}
				\end{equation}
				where $\Im$ represents the image part of a complex fuction. Multiplying both sides of the last equation of \eqref{fully1} by $2\tau r^{n+\frac{1}{2}}$ to get
				\begin{equation}\label{conservation fully 4}
								(r^{n+1})^2 - (r^n)^2 = 2\tau r^{n+\frac{1}{2}}  ( f_\mathcal{N}(\widetilde{u}^{n+\frac{1}{2}}), \mathcal{A}(u^n, u^{n+1}, v^n, v^{n+1}) )_{l^2}.
				\end{equation}
				Taking the real part on both sides of \eqref{conservation fully 2}, then using \eqref{conservation fully 3}, \eqref{conservation fully 4}, we obtain the desired result.
\end{proof}
Besides the conservative property, a remarkable feature of the scheme \eqref{fully1} is that it can be implemented explicitly. Let 
\begin{equation*}
\left\lbrace
\begin{aligned}
				u_1^{n+1} &= e^{11}(\tau) u^n + e^{12}(\tau)v^n,  \\
				v_1^{n+1} &= e^{12}(\tau) u^n + e^{22}(\tau)v^n,  
\end{aligned}
\right. \ \text{and}
\left\lbrace
\begin{aligned}
				u_2^{n+1} &=  - \tau \beta e^{12}\big(\tfrac{\tau}{2}\big) f_\mathcal{N}(\widetilde{u}^{n+\tfrac{1}{2}}), \\
				v_2^{n+1} &=  - \tau \beta e^{22}\big(\tfrac{\tau}{2}\big) f_\mathcal{N}(\widetilde{u}^{n+\tfrac{1}{2}}).
\end{aligned}
\right.
\end{equation*}
From \eqref{fully1}, $u^{n+1}$ can be regarded as the linear combination of $u_1^{n+1}$ and $u_2^{n+1}$, with respect to $r^{n+\frac{1}{2}}$ as follows:
\begin{equation}\label{split-u}
				u^{n+1} = u_1^{n+1} + r^{n+\frac{1}{2}}u_2^{n+1}. \ 
\end{equation}
Analogously, we have 
\begin{equation}\label{split-v}
				v^{n+1} = v_1^{n+1} + r^{n+\frac{1}{2}}v_2^{n+1}. \ 
\end{equation}
Inserting \eqref{split-u} and \eqref{split-v} into the third equation of \eqref{fully1}, it is readily to deduce that $r^{n+\frac{1}{2}}$ can be updated by
\begin{equation}\label{update-r}
				r^{n+\frac{1}{2}} = \frac{4 r^n + \tau b_1^{n+\frac{1}{2}}}{4 - \tau b_2^{n+\frac{1}{2}}},
\end{equation}
where 
\begin{equation*}
\begin{aligned}
				b_1^{n+\frac{1}{2}} &= \Re  \big(f_\mathcal{N}(\widetilde{u}^{n+\frac{1}{2}}), \mathcal{A}(u^n, u_1^{n+1}, v^n, v_1^{n+1})\big)_{l^2}  , \\
				b_2^{n+\frac{1}{2}} &= \Re  \big(f_\mathcal{N}(\widetilde{u}^{n+\frac{1}{2}}), \mathcal{A}(0, u_2^{n+1}, 0, v_2^{n+1})\big)_{l^2} . 
\end{aligned}
\end{equation*}
In conjunction with \eqref{split-u}, \eqref{split-v} and \eqref{update-r}, we immediately obtain $u^{n+1}$ and $v^{n+1}$.

\section{Error estimates for the 2D NLSW}\label{sec.3}
In this section, we present optimal $H^1$ error estimates for the fully discrete scheme \eqref{fully1}. We first introduce some auxiliary notations and lemmas, then present and prove the main results.
\subsection{Auxiliary lemmas and main results}

 We recall the conventional Sobolev space $H^m(\Omega)$ and $H_0^1(\Omega)$. Define the subspace of $H^m(\Omega) \cap H_0^1(\Omega)$ as $H_s^m(\Omega) = \{u \in H^m(\Omega)| \partial_x^{2k}u(x,y) = \partial_y^{2k}u(x,y) = 0, \ (x,y) \in \partial \Omega,  k\in \mathbb{N}, 0\leq 2k \leq m\} $ (the boundary values are understood in the trace sense). Specifically, we denote $L_s^2(\Omega) = H_s^0(\Omega)$. Suppose that $u(x, y)$ can be expanded into a sine series, such that 
\begin{equation}\label{series_u_infty}
				u(x, y) = \sum_{p= 1}^{+\infty}\sum_{q=1}^{+\infty} \breve{u}_{pq} \sin{\left(\mu_p (x - x_L)\right)}\sin{\left(\nu_q( y - y_L )\right)}.
\end{equation}
 Then, the norm of space $H_s^m(\Omega)$ can be characterized by the sine frequency as
\begin{equation*}
				\|u\|_{H_s^m} = \Big( \sum\limits_{p=1}^\infty \sum\limits_{q=1}^\infty \lambda_{pq}^{2m} |\breve{u}_{pq}|^2 \Big)^\frac{1}{2},
\end{equation*}
which is equivalent to the classical $H^m$ norm in this subspace. Let $V$ be a Banach space, we shall also use the standard notation $L^p(0, T; V)$ and $\|\cdot\|_{L^p(0,T;V)}$ for $\ p = 1,\cdots, \infty$ to represent the Bochner space and the corresponding norm. In the following derivations, we denote $C$ as a generic positive constant independent of the discretization parameters. 

Define the discrete space $X_\mathcal{N}$ where the numerical solutions located in as follows.
\begin{equation*}
				X_\mathcal{N} = {\rm span}\{ \sin{\left( \mu_p (x - x_L) \right)} \sin{\left(\nu_q(y - y_L)\right)} | (x, y) \in \Omega, \ 1 \leq p,q \leq N-1  \}.
\end{equation*}
The orthogonal projection of $u(x, y) \in L_s^2(\Omega)$ to $X_\mathcal{N}$ is
\begin{equation*}
				\varPi_\mathcal{N}u (x, y) =  \sum\limits_{p=1}^{N-1} \sum\limits_{q=1}^{N-1} \breve{u}_{pq} \sin{\left(\mu_p (x - x_L)\right)} \sin{\left( \nu_q (y - y_L) \right)},
\end{equation*}
which is just the truncation of the infinite series \eqref{series_u_infty}. Furthermore, suppose that $u(x, y)$ can be defined in the pointwise sense, we also introduce the interpolation operator $I_\mathcal{N}: C_0(\overline{\Omega}) \to X_\mathcal{N}$ as follows
\begin{equation*}
				I_\mathcal{N}u(x, y) = \sum\limits_{p=1}^{N-1}\sum\limits_{q=1}^{N-1} \widehat{u}_{pq} \sin{\left(\mu_p (x - x_L)\right)} \sin{\left( \nu_q (y - y_L) \right)},
\end{equation*}
where $C_0(\overline{\Omega})$ represents the space of continuous functions with zero boundaries on $\overline{\Omega}$, and
\begin{equation*}
				\widehat{u}_{pq} = \frac{4}{N^2}\sum\limits_{j=1}^{N-1} \sum\limits_{k=1}^{N-1} u(x_j, y_k) \sin{\left( \mu_p(x_j - x_L) \right)} \sin{\left(\nu_q (y_k - y_L)\right)}, \quad (p, q)  \in \mathcal{T}_\mathcal{N}.
\end{equation*}
We note here the definition of the interpolation operator $I_\mathcal{N}$ can also be extended to the grid functions as \eqref{interp-u}. The standard approximation and stability properties of the projection and interpolation operators are provided below.
\begin{lem}[\cite{NLSW-SPS,GP-SPS}]\label{approx-prop-continuous}
				For any $0 \leq k \leq m$ and $u \in H_s^m(\Omega)$, we have
				\begin{equation*}
				\begin{aligned}
								&\|u - \varPi_\mathcal{N} u\|_{H_s^k} \leq C N^{k-m} \|u\|_{H_s^m}, \quad \|\varPi_\mathcal{N} u\|_{H_s^m} \leq C\|u\|_{H_s^m}, \\
								&\|u - I_\mathcal{N} u\|_{H_s^k} \leq C N^{k-m} \|u\|_{H_s^m}, \ \quad \|I_\mathcal{N} u\|_{H_s^m} \leq C\|u\|_{H_s^m}. \\
				\end{aligned}
				\end{equation*}
\end{lem}

Next, we introduce the discrete semi $H^1$ and semi $H^2$ norms with respect to the finite difference methods.  For any $u \in \mathcal{M}_\mathcal{N}$, we introduce the first and second-order difference quotients as follows:
\begin{equation*}
				\delta_x u_{jk} = \frac{u_{j+1k} - u_{jk}}{h_1}, \quad \delta_x^2 u_{jk} = \frac{u_{j+1k} - 2u_{jk} + u_{j-1k}}{h_1^2}.
\end{equation*}
The definitions of $\delta_y u^n_{jk}$ and $\delta_y^2 u^n_{jk}$ are analogous. The discrete semi norms with respect to difference quotients of $u$ are
\begin{equation*}
\begin{aligned}
				 &|u|_{h,1} = \Big( h_1h_2\sum\limits_{j=0}^{N-1}\sum\limits_{k=1}^{N-1}|\delta_x u_{jk}|^2 + h_1h_2\sum\limits_{j=1}^{N-1}\sum\limits_{k=0}^{N-1}|\delta_y u_{jk}|^2 \Big)^{\frac{1}{2}}, \quad |u|_{h,2} = \Big(\|\delta^2_x u\|^2_{l^2} + \|\delta^2_y u\|^2_{l^2} \Big)^{\frac{1}{2}}.
\end{aligned}
\end{equation*}

The following lemmas are necessary for the forthcoming analyses. 
\begin{lem}[Norm equivalence \cite{GP-SPS}]\label{norm-equivalence2d}
For any $u \in \mathcal{M}_\mathcal{N}$, we have
				\begin{equation*}
								|u|_{h,1} \leq |u|_{\mathcal{N},1} \leq \frac{\pi}{2}|u|_{h,1}, \ |u|_{h,2} \leq |u|_{\mathcal{N},2} \leq \frac{\pi^2}{4} |u|_{h,2}.
				\end{equation*}
\end{lem}
\begin{lem}[\cite{GP-SPS}]\label{approximation-prop-discrete}
				Suppose that $u \in H_s^m(\Omega)$ with $m \geq2$, we have
				\begin{equation*}
								\|u - \varPi_\mathcal{N} u \|_{l^2} + N^{-1} |u - \varPi_\mathcal{N} u|_{\mathcal{N}, 1} + N^{-2} |u - \varPi_\mathcal{N} u|_{\mathcal{N}, 2} \leq  C\|u\|_{H_s^m}N^{-m}.
				\end{equation*}
\end{lem}
\begin{lem}[Discrete Poincar\'e inequality \cite{Poincare}]\label{Poincare}
				For any $u\in \mathcal{M}_\mathcal{N}$, we have $\|u\|_{l^2} \leq C |u|_{h,1}$.
\end{lem}
\begin{lem}[Discrete Sobolev inequality I \cite{GP-SPS,Wang-NLS,Wang-GL}]\label{Sobolev_embedding2d}
			For any $u \in \mathcal{M}_\mathcal{N}$ and $p \in [2, \infty)$, we have 
			\begin{equation}
								\|u\|_{l^p} \leq \|u\|_{l^2}^{\frac{2}{p}} (C_p|u|_{h,1} + \frac{1}{l}\|u\|_{l^2})^{1-\frac{2}{p}},
			\end{equation}
			where $C_p = \max\{2\sqrt{2}, \frac{p}{\sqrt{2}}\}$ and $l = \min{\{X, Y\}}$.
\end{lem}
\begin{lem}[Discrete Sobolev inequality II \cite{GP-SPS}] \label{Sobolev_embedding2d_Linfty}
				For any $u \in \mathcal{M}_\mathcal{N}$, we have
				\begin{equation*}
								\|u\|_{l^\infty} \leq \|u\|_{l^2}^{1-\frac{d}{4}}(|u|_{h,2} + \|u\|_{l^2})^{\frac{d}{4}}.
				\end{equation*}
\end{lem}
\begin{lem}\label{norm_equal}
				For any $u \in X_\mathcal{N} \cap H_s^m(\Omega)$, we have
				\begin{equation*}
								\|u\| = \|u\|_{l^2}, \ \|u\|_{H_s^1} = |u|_{\mathcal{N},1}, \ \|u\|_{H_s^2} = |u|_{\mathcal{N},2}.
				\end{equation*}
\end{lem}
\begin{lem}[Estimates of $e^{\mu \nu}(t)$]\label{spectral-radius}
				Suppose that $A: \mathcal{M}_\mathcal{N} \to \mathcal{M}_\mathcal{N}$ is linear, let $\ssr A \ssr:= \sup\limits_{u \in \mathcal{M}_\mathcal{N} \backslash \{0\} } \frac{\|Au\|_{l^2}}{\|u\|_{l^2}} $ be it spectral radius, the following estimates hold.
				\begin{equation*}
				\begin{aligned}
								&\ssr e^{11}(t) \ssr \leq 1, \quad \ssr e^{12}(t) \ssr \leq C, \quad \ssr e^{22}(t) \ssr \leq 1, \\
								&\|\Delta_\mathcal{N} e^{12}(t)u\|_{l^2} = \|e^{21}(t)u\|_{l^2} \leq |u|_{\mathcal{N},1}, \quad |e^{12}(t) u|_{\mathcal{N},1} \leq \frac{1}{2}\|u\|_{l^2}, \quad \forall u \in \mathcal{M}_\mathcal{N}.
				\end{aligned}
				\end{equation*}
\end{lem}
\begin{proof}
				We only prove the first inequality, the rest are analogous. Proposition \ref{exp-element2d} and Lemma \ref{norm_equal} give
				\begin{equation*}
				\begin{aligned}
								\| e^{11}(t) u \|_{l^2}^2 &= \| e^{11}(t) I_\mathcal{N} u \|_{l^2}^2  \\
																					&=\sum\limits_{p=1}^{N-1}\sum\limits_{q=1}^{N-1} |\widehat{e}^{11}_{pq}(t)|^2 |\widehat{u}_{pq}|^2 \\
																			&= \sum\limits_{p=1}^{N-1}\sum\limits_{q=1}^{N-1} \frac{\alpha^2 + 4\lambda_{pq}^2 \cos^2{(\omega_{pq} t /2)}}{\omega_{pq}^2} |\widehat{u}_{pq}|^2\\ 
																			&\leq \sum\limits_{p=1}^{N-1}\sum\limits_{q=1}^{N-1} |\widehat{u}_{pq}|^2 = \|u\|_{l^2}^2,
				\end{aligned}
				\end{equation*}
				where $\widehat{u}_{pq}$ is the discrete sine transform of $u$ \eqref{interp-u}. By the definition of $\ssr e^{11}(t) \ssr$, we can obtain the desired result.
\end{proof}
\begin{lem}\label{nonlinear-lem}
				For any $u, v, a, b \in \mathcal{M}_\mathcal{N}$, we have the following inequalities
				\begin{enumerate}
				\item $\|\mathcal{A} (u, v, a, b) \|_{l^2} \leq \frac{1}{2}(|u|_{\mathcal{N},1} + |v|_{\mathcal{N},1} + \|a\|_{l^2} + \|b\|_{l^2})$.
				\item $\|f_\mathcal{N} (u) - f_\mathcal{N} (v)\|_{l^2} \leq C (\| g(u) - g(v)\|_{l^2} + \| G(u) - G(v)\|_{l^2} \|g(v)\|_{l^2} )$.
\end{enumerate}
\end{lem}
\begin{proof}
				The first inequality can be obtained by using the triangular inequality and Lemma \ref{spectral-radius}. For the second inequality, we add and subtract some intermediate terms to get 
				\begin{equation*}
								f_\mathcal{N} (u) - f_\mathcal{N} (v) = \frac{g(u) - g(v)}{\sqrt{H_\mathcal{N} (u)}} + \frac{(H_\mathcal{N}(v) - H_\mathcal{N}(u)) g(v)}{\sqrt{H_\mathcal{N}(u) H_\mathcal{N}(v)}(\sqrt{H_\mathcal{N} (u)} + \sqrt{H_\mathcal{N}(v)})}.
				\end{equation*}
				The result of Lemma \ref{nonlinear-lem} is straightforward after using the triangular and the Cauchy-Schwarz inequalities.
\end{proof}
Now, we are in the position to establish the error estimates. For the sake of simplicity, we use the capital letter $U(x, y, t)$ and $\{U(t)\}_{jk} = U(x_j, y_k, t)$ to represent the exact solutions of the system \eqref{sav} in the subsequent derivations. The convergence results are described as follows.
\begin{thm}[Main theorem for $\beta > 0$]\label{main-thm}
				Suppose that $U(\cdot) \in L^\infty(0,T; H_s^m(\Omega))$, $\partial_t U(\cdot) \in L^\infty(0,T;$ $H_s^{m-1}(\Omega))$ and $\partial_{tt}U(\cdot) \in L^\infty(0,T; L^2(\Omega))$ with $m \geq 2$. There exists a sufficiently small constant $\tau_0$, independent of $\tau$ and $N$, such that for any $0 < \tau \leq \tau_0$ and $0 \leq n \leq M$, the following estimates hold
				\begin{equation*}
								\|U(t_n) - u^n\|_{\mathcal{N}, 1} + \|V(t_n) - v^n\|_{l^2} +  |R(t_n) - r^n| \leq C(N^{-m+1} + \tau^2).
				\end{equation*}
\end{thm}

\begin{thm}[Main theorem for $\beta < 0$]\label{main-thm2}
				Suppose that $U(\cdot) \in L^\infty(0,T; H_s^m(\Omega))$, $\partial_t U(\cdot) \in L^\infty(0,T;$ $H_s^{m-1}(\Omega))$ and $\partial_{tt}U(\cdot) \in L^\infty(0,T; L^2(\Omega))$ with $m \geq 3$. There exists a sufficiently small constant $\tau_0$ and a sufficiently large constant $N_0$, independent of $\tau$ and $N$, such that for any $0 < \tau \leq \tau_0$, $N>N_0$ and $0 \leq n \leq M$, the following estimates hold
				\begin{equation*}
								\|U(t_n) - u^n\|_{\mathcal{N}, 1} + \|V(t_n) - v^n\|_{l^2} +  |R(t_n) - r^n| \leq C(N^{-m+1} + \tau^2).
				\end{equation*}
				as well as the boundedness of $u^n$ 
				\begin{equation*}
								\|u^n\|_{l^\infty} \leq M_1 + 1, \ M_1 = \max \{ \|U(\cdot)\|_{L^\infty(0, T; L^\infty(\Omega))}, \|\varPi_\mathcal{N}U(\cdot)\|_{L^\infty(0, T; L^\infty(\Omega))}  \}.
				\end{equation*}
\end{thm}
\begin{rmk}
				We will prove Theorem \ref{main-thm} by using a smilar technique developed in \cite{Wang-NLS}. As for Theorem \ref{main-thm2}, we introduce an improved induction argument to prove it. Comparing Theorem \ref{main-thm} with Theorem \ref{main-thm2}, the latter requires stronger regularity hypotheses to $u$, and a restriction on the spatial step is also introduced, which are used to recover the $L^\infty$ boundedness of $u^n$ in the mathematical induction.
\end{rmk}
\subsection{Proof of the main result}
\textbf{Step1.} We first project both sides of \eqref{sav} into $X_\mathcal{N}$, then recast the obtained system into a form similar to \eqref{fully1}. Let
\begin{equation*}
				U^\star_{jk}(t) = \varPi_\mathcal{N}U(x_j, y_k, t), \ f^\star_{jk}(t) = \varPi_\mathcal{N}f(U)(x_j, y_k, t), \ (f^\star_\mathcal{N})_{jk} = \varPi_\mathcal{N}f_\mathcal{N}(U)(x_j, y_k, t). 
\end{equation*}
Acting the $L^2$ projection operator on both sides of \eqref{sav} and notice that $\varPi_\mathcal{N}$ commutes with the differential operator in the sense
\begin{equation*}
				  \varPi_\mathcal{N} \Delta U(x_j, y_k, t) = \Delta \varPi_\mathcal{N}U(x_j, y_k, t) = \Delta_\mathcal{N} U^\star_{jk},
\end{equation*}
we have
\begin{equation}\label{sav-proj}
\left\lbrace
\begin{aligned}
				&\dot{U}^\star(t) = V^\star(t), \ \dot{V}^\star(t) = \Delta_\mathcal{N} U^\star(t) - i\alpha V^\star(t) - \beta R(t) f^\star(U(t)), \\ 
				&\dot{R}(t) = \Re (f^\star(U)(t), V^\star(t))_{l^2} + D_1(t), \\ 
\end{aligned}
\right.
\end{equation}
where $D_1$ is the difference between the continuous and the discrete inner products, i.e.,
\begin{equation*}
				D_1 = \Re (f(U(t)), V(t)) - \Re (f^\star(U(t)), V^\star(t))_{l^2}.
\end{equation*}
Integrating the first two equations of \eqref{sav-proj} around $t = t_n$ by the variation-of-constant formula gives
\begin{equation}\label{proj-exact}
\begin{aligned}
				U^{\star}(t_n + s) &= e^{11}(s) U^{\star}(t_n) + e^{12}(s) V^{\star}(t_n) - \beta \int_0^s R(t_n + \sigma) e^{12}(s-\sigma) f^\star(U(t_n+\sigma)) d\sigma, \\
				V^{\star}(t_n + s) &= e^{21}(s) U^\star (t_n) + e^{22}(s) V^\star(t_n) - \beta \int_0^s R(t_n + \sigma) e^{22}(s-\sigma) f^\star(U(t_n+\sigma)) d\sigma,
\end{aligned}
\end{equation}
Let $s = \frac{\tau}{2}$ in the second equation of \eqref{proj-exact} and recall $R(t) = \sqrt{H[U(t)]}$, we obtain
\begin{equation}\label{proj-mid1}
				V^{\star}(t_n + \tfrac{\tau}{2})= e^{21}\big(\tfrac{\tau}{2}\big) U^\star (t_n) + e^{22}\big(\tfrac{\tau}{2}\big) V^{\star}(t_n) - \beta \int_0^{\tfrac{\tau}{2}} e^{22}\big(\tfrac{\tau}{2}-\sigma\big) f^{\star}(U(t_n+\sigma)) d\sigma.
\end{equation}
Analogously, we can integrate the system \eqref{sav-proj} from $t_{n+1}$ to $t_{n} + \tfrac{\tau}{2}$ and get
\begin{equation}\label{proj-mid2}
				V^{\star}\big(t_n + \tfrac{\tau}{2}\big) = e^{21}\big(-\tfrac{\tau}{2}\big) U^\star(t_{n+1}) + e^{22}\big(-\tfrac{\tau}{2}\big) V^{\star}(t_{n+1}) + \beta \int_0^{\tfrac{\tau}{2}} e^{22}(-\sigma) f^{\star}(U\big(t_n + \sigma + \frac{\tau}{2}\big)) d\sigma.
\end{equation}
Adding \eqref{proj-mid1} and \eqref{proj-mid2} together yields 
\begin{equation*}
				V^\star (t_n + \tfrac{\tau}{2}) = \mathcal{A}(U^{\star}(t_n), U^{\star}(t_{n+1}), V^{\star}(t_n), V^{\star}(t_{n+1})) + D_2^{n+\frac{1}{2}},
\end{equation*}
where
\begin{equation}\label{rho-2}
				D_2^{n+\frac{1}{2}} = \frac{\beta}{2} \int_0^{\tfrac{\tau}{2}} e^{22}(-\sigma) f^\star (U(t_n+\sigma + \tfrac{\tau}{2})) -  e^{22}\big(\tfrac{\tau}{2}-\sigma\big) f^\star (U(t_n + \sigma))  d\sigma.
\end{equation}
Based on the above preparations, we recast the third equation of \eqref{sav-proj} at $t_{n+\frac{1}{2}}$ into
\begin{equation}\label{exact-half}
				\partial_t R\big(t_n + \tfrac{\tau}{2}\big) = \Re (f^\star(u\big(t_n + \tfrac{\tau}{2}\big)), \mathcal{A}(U^{\star}(t_n), U^{\star}(t_{n+1}), V^{\star}(t_n), V^{\star}(t_{n+1}) + D_2^{n+\frac{1}{2}})_{l^2} + D_1^{n+\frac{1}{2}},
\end{equation}
where 
\begin{equation*}
				D_1^{n+\tfrac{1}{2}} = \Re (f(U\big(t_n+ \tfrac{\tau}{2})\big), V\big(t_n + \tfrac{\tau}{2}\big) ) - \Re \big(f^\star\big(U^\star(t_n + \tfrac{\tau}{2})\big), V^{\star}(t_n + \tfrac{\tau}{2})\big)_{l^2}.
\end{equation*}
\textbf{Step2.} Then we establish the estimates of the local errors. The projection solutions $U^\star(\cdot)$, $V^\star(\cdot)$ and the exact solution of $R(t)$ can be regarded as satisfying \eqref{fully1} together with the truncation errors $R_1^{n+1}, R_2^{n+1}$ and $R_3^{n+\frac{1}{2}}$, such that
\begin{equation}\label{local-error}
\begin{aligned}
				U^{\star}(t_{n+1}) &= e^{11}(\tau) U^{\star}(t_n) + e^{12}(\tau) V^\star(t_n) - \tau \beta R^{n+\frac{1}{2}} e^{12}\big(\tfrac{\tau}{2}\big) f_\mathcal{N}(\widetilde{U}^{\star n+\frac{1}{2}}) + R_1^{n+1}, \\
				V^{\star}(t_{n+1}) &= e^{21}(\tau) U^{\star}(t_n) + e^{22}(\tau) V^\star(t_n) - \tau \beta R^{n+\frac{1}{2}} e^{22}\big(\tfrac{\tau}{2}\big) f_\mathcal{N}(\widetilde{U}^{\star n+\frac{1}{2}}) + R_2^{n+1}, \\
				\delta_t R^{n+\frac{1}{2}}	&= \Re  ( f_\mathcal{N}(\widetilde{U}^{\star n+\frac{1}{2}}), \mathcal{A}(U^{\star}(t_n), U^{\star}(t_{n+1}), V^{\star}(t_n), V^{\star}(t_{n+1})) )_{l^2}  + R_3^{n+\frac{1}{2}},
\end{aligned}
\end{equation}
where 
\begin{equation*}
				\widetilde{U}^{\star n+\frac{1}{2}} = \frac{3 U^\star(t_n) - U^\star(t_{n-1})}{2}, \quad R^{n+\frac{1}{2}} = \frac{R(t_n) + R(t_{n+1})}{2}, \quad \delta_t R^{n+\frac{1}{2}} = \frac{R(t_{n+1}) - R(t_n)}{\tau}.
\end{equation*}
The estimates of the local errors are provided below.
\begin{lem}\label{lem local-error}
Suppose the hypotheses of Theorem \ref{main-thm}, \ref{main-thm2} hold. Then we have the following estimates to the local errors
$R_1^n$, $R_2^n$ and $R_3^{n+\frac{1}{2}}$. For $n = 1, \cdots, M-1$, we have
				\begin{equation*}
								\big| R_1^{n+1} \big|^2_{\mathcal{N},\gamma} \leq C\tau^2(\tau^4 +  N^{-2m + 2\gamma - 2}), \big|R_2^{n+1}\big|^2_{\mathcal{N},\gamma-1} \leq C\tau^2(\tau^4 +  N^{-2m+2\gamma}), \big| R_3^{r,n+\frac{1}{2}} \big|^2 \leq C (\tau^4 + N^{-2m}),
				\end{equation*}
				and for $n=0$, we have
				\begin{equation*}
								\big| R_1^1 \big|^2_{\mathcal{N},\gamma} \leq C\tau^2(\tau^2 +  N^{-2m + 2\gamma -2}), \big|R_2^1\big|^2_{\mathcal{N},\gamma-1} \leq C\tau^2(\tau^2 +  N^{-2m+2\gamma}), \big| R_3^{\frac{1}{2}} \big|^2 \leq C (\tau^2 + N^{-2m}).
				\end{equation*}
\end{lem}
Here, $\gamma = 1,2$. Due to the limitation of space, we will leave the proof of Lemma \ref{lem local-error} to the Appendix \ref{A}.

\textbf{Step3.}  We next prove Theorem \ref{main-thm}, \ref{main-thm2}, respectively. To further simplify the notations, we denote by 
\begin{equation*}
				\widetilde{f}_\mathcal{N}^{n+\frac{1}{2}} = f_\mathcal{N}(\widetilde{u}^{n+\frac{1}{2}}), \quad \widetilde{f}_\mathcal{N}^{\star n+\frac{1}{2}} = f_\mathcal{N}(\widetilde{U}^{\star n+\frac{1}{2}}),
\end{equation*}
and $\widetilde{G}^{n+\frac{1}{2}}, \widetilde{G}^{\star n+\frac{1}{2}}, \widetilde{g}^{n+\frac{1}{2}}, \widetilde{g}^{\star n+\frac{1}{2}}$ are defined analogously. Define the solution errors
\begin{equation*}
				\xi^n = U^{\star}(t_n) - u^n, \quad \eta^n = V^{\star}(t_n) - v^n, \quad \zeta^n = R(t_n) - r^n,
\end{equation*}
which satisfy the following equations by subtracting \eqref{fully1} and \eqref{local-error}.
 \begin{equation}\label{error-eq}
				\left\lbrace
\begin{aligned}
				\xi^{n+1} &= e^{11}(\tau) \xi^n + e^{12}(\tau) \eta^n - \tau \beta e^{12}\big(\tfrac{\tau}{2}\big) E_2^{n+\frac{1}{2}} + R_1^{n+1}, \\ 
				\eta^{n+1} &= e^{21}(\tau) \xi^n + e^{22}(\tau) \eta^n - \tau \beta e^{22}\big(\tfrac{\tau}{2}\big) E_2^{n+\frac{1}{2}} + R_2^{n+1}, \\
				\delta_t \zeta^{n+\frac{1}{2}} &= E_1^{n+\frac{1}{2}} + R_3^{n+\frac{1}{2}},
\end{aligned}
				\right.
\end{equation}
where $E_1^{n+\frac{1}{2}}$ and $E_2^{n+\frac{1}{2}}$ are
\begin{equation*}
\begin{aligned}
				E_1^{n+\frac{1}{2}} &=  \Re( \widetilde{f}_\mathcal{N}^{\star,n+\frac{1}{2}}, \mathcal{A}(U^{\star}(t_n), U^{\star}(t_{n+1}), V^{\star}(t_n), V^{\star}(t_{n+1}) ) )_{l^2}  \\
														&\quad - \Re( \widetilde{f}_\mathcal{N}^{n+\frac{1}{2}}, \mathcal{A}(u^n,u^{n+1},v^n, v^{n+1}) )_{l^2} , \\ 
\end{aligned}
\end{equation*}
\begin{equation*}
				E_2^{n+\frac{1}{2}} = R^{n+\frac{1}{2}} \widetilde{f}_\mathcal{N}^{\star,n+\frac{1}{2}} -  r^{n+\frac{1}{2}} \widetilde{f}_\mathcal{N}^{n+\frac{1}{2}}.
\end{equation*}
\textbf{I. Convergence for the case $\beta > 0$.} We first prove the result for $\beta > 0$. Since all the terms in the discrete energy \eqref{conservation fully 1} are non-negative, we can obtain a priori estimates from the discrete conservation law Theorem \ref{fully-conservation}.
\begin{lem}[A priori estimate]\label{priori-estimate2d}
				There exists a constant $M_2$, such that
				\begin{equation*}
								\max\limits\{|u^n|_{\mathcal{N},1}, \|u^n\|_{l^p}, \|v^n\|_{l^2}, |r^n|\} \leq M_2.
				\end{equation*}
\end{lem}
We remark here the boundedness $\|u^n\|_{l^p}$ can be obtained by combining Lemma \ref{norm-equivalence2d}, \ref{Poincare}, \ref{Sobolev_embedding2d}. Consequently, we can introduce the following lemma to bound the nonlinear terms.
\begin{lem}\label{sigma-l2}
				The following estimates for $E_1^{n+\frac{1}{2}}$ and $E_2^{n+\frac{1}{2}}$ hold
\begin{equation*}
\begin{aligned}
				|E_1^{n+\frac{1}{2}}|^2 &\leq C(|\xi^{n-1}|_{\mathcal{N},1}^2 + |\xi^n|_{\mathcal{N},1}^2  + \|\eta^n\|_{l^2}^2 + \|\eta^{n+1}\|_{l^2}^2), \\
				\|E_2^{n+\frac{1}{2}}\|_{l^2}^2 &\leq C(\|\xi^{n-1}\|_{l^2}^2 + \|\xi^n\|_{l^2}^2 + |\zeta^n|^2 + |\zeta^{n+1}|^2).
\end{aligned}
\end{equation*}
\end{lem}
\begin{proof}
				According to the Cauchy-Schwarz inequality and Lemma \ref{nonlinear-lem}, there is
				\begin{equation}\label{sigma1-trans}
				\begin{aligned}
								|E_1^{n+\frac{1}{2}}|^2 &\leq 2| \Re (\widetilde{f}_\mathcal{N}^{\star n+\frac{1}{2}} - \widetilde{f}_\mathcal{N}^{n+\frac{1}{2}}, \mathcal{A}(U^{\star}(t_n), U^{\star}(t_{n+1}), V^{\star} (t_n), V^{\star}(t_{n+1})))_{l^2}|^2\\
																						 &+2| \Re (\widetilde{f}_\mathcal{N}^{n+\frac{1}{2}}, \mathcal{A}(\xi^n, \xi^{n+1}, \eta^n, \eta^{n+1}))_{l^2}|^2  \\
																	 &\leq C \|\widetilde{u}^{n+\frac{1}{2}}\|_{l^6}^6 \|\widetilde{G}^{\star n+\frac{1}{2}}-\widetilde{G}^{n+\frac{1}{2}}\|_{l^2}^2 
																	 + C \|\widetilde{g}^{\star n+\frac{1}{2}}-\widetilde{g}^{n+\frac{1}{2}}\|^2_{l^2}  \\ 
																	 & +  C \|\widetilde{u}^{n+\frac{1}{2}}\|_{l^6}^6 \|\mathcal{A}(\xi^n, \xi^{n+1}, \eta^n, \eta^{n+1})\|_{l^2}^2 \\ 
																	 & = E_{11}^{n+\frac{1}{2}} + E_{12}^{n+\frac{1}{2}} + E_{13}^{n+\frac{1}{2}} \\ 
				\end{aligned}
				\end{equation}
				From Lemma \ref{Sobolev_embedding2d}, \ref{priori-estimate2d}, we have 
				\begin{equation*}
				\begin{aligned}
								E_{11}^{n+\frac{1}{2}} &\leq C (|\xi^n|_{\mathcal{N},1}^2 + |\xi^{n+1}|_{\mathcal{N},1}^2 + \|\eta^n\|^2_{l^2} + \|\eta^{n+1}\|^2_{l^2}), \\
								E_{12}^{n+\frac{1}{2}} &\leq C \|\widetilde{G}^{\star n+\frac{1}{2}}-\widetilde{G}^{n+\frac{1}{2}}\|_{l^2}^2.
				\end{aligned}
				\end{equation*}
As an example, we demonstrate how to bound $\|\widetilde{G}^{\star n+\frac{1}{2}}-\widetilde{G}^{n+\frac{1}{2}}\|_{l^2}^2$. In view of its definition, we can derive
\begin{equation}\label{expandG-G}
				\begin{aligned}
								\widetilde{G}^{\star,n+\frac{1}{2}}_{jk}-\widetilde{G}^{n+\frac{1}{2}}_{jk} &= \widetilde{\xi}^{n+\frac{1}{2}}_{jk} |\widetilde{U}^{\star,n+\frac{1}{2}}_{jk}|^2 \overline{\widetilde{U}}_{jk}^{\star n+\frac{1}{2}} + \overline{\widetilde{\xi}}_{jk}^{n+\frac{1}{2}} |\widetilde{U}_{jk}^{\star n+\frac{1}{2}}|^2 \widetilde{u}_{jk}^{n+\frac{1}{2}} \\
								&= \widetilde{\xi}^{n+\frac{1}{2}}_{jk} |\widetilde{U}^{\star,n+\frac{1}{2}}_{jk}|^2 \overline{\widetilde{U}}_{jk}^{\star n+\frac{1}{2}} + \overline{\widetilde{\xi}}_{jk}^{n+\frac{1}{2}} |\widetilde{U}_{jk}^{\star n+\frac{1}{2}}|^2 \widetilde{u}_{jk}^{n+\frac{1}{2}} \\
								&+ \widetilde{\xi}_{jk}^{n+\frac{1}{2}} \overline{\widetilde{U}}_{jk}^{\star n+\frac{1}{2}} |\widetilde{u}_{jk}^{n+\frac{1}{2}}|^2 
								+ \overline{\widetilde{\xi}}_{jk}^{n+\frac{1}{2}} |\widetilde{u}_{jk}^{n+\frac{1}{2}}|^2 \widetilde{u}_{jk}^{n+\frac{1}{2}} \\ 
								&= \{P_1 \}_{jk} + \{P_2 \}_{jk} + \{P_3 \}_{jk} + \{P_4\}_{jk}.
				\end{aligned}
\end{equation}
Using the identity $u^n = U^{\star n} - \xi^n$ to further expand $P_\nu \ (\nu=2,3,4)$, we get
\begin{equation*}
\begin{aligned}
				\{P_2\}_{jk} &= \overline{\widetilde{\xi}}^{n+\frac{1}{2}}_{jk}| \widetilde{U}^{\star n+\frac{1}{2}}_{jk}| \widetilde{U}^{\star n+\frac{1}{2}}_{jk} - |\widetilde{\xi}_{jk}^{n+\frac{1}{2}}|^2 | \widetilde{U}^{\star n+\frac{1}{2}}_{jk}|^2, \\
				\{P_3\}_{jk} &= \widetilde{\xi}_{jk}^{n+\frac{1}{2}} | \widetilde{U}^{\star n+\frac{1}{2}}_{jk}| \overline{\widetilde{U}}^{\star n+\frac{1}{2}}_{jk} - ( \widetilde{\xi}_{jk}^{n+\frac{1}{2}} )^2 (\overline{\widetilde{U}}^{\star n+\frac{1}{2}}_{jk})^2, \\ 
										 & \quad - | \widetilde{\xi}_{jk}^{n+\frac{1}{2}}|^2| \widetilde{U}^{\star n+\frac{1}{2}}_{jk}|^2 + |\widetilde{\xi}_{jk}^{n+\frac{1}{2}}|^2 \widetilde{\xi}_{jk}^{n+\frac{1}{2}} \overline{\widetilde{U}}^{\star n+\frac{1}{2}}_{jk}, \\
				\{P_4\}_{jk} &= \widetilde{\xi}_{jk}^{n+\frac{1}{2}} |\widetilde{U}^{\star n+\frac{1}{2}}_{jk}|^2 \overline{\widetilde{U}}^{\star,n+\frac{1}{2}}_{jk} - 2( \widetilde{\xi}_{jk}^{n+\frac{1}{2}} )^2 | \widetilde{U}^{\star n+\frac{1}{2}}_{jk}|^2 \\ 
										 & \quad - |\widetilde{\xi}_{jk}^{n+\frac{1}{2}} |^2 (\widetilde{U}^{\star n+\frac{1}{2}}_{jk})^2 - ( \widetilde{\xi}_{jk}^{n+\frac{1}{2}})^3 \widetilde{U}^{\star n+\frac{1}{2}}_{jk} - |\widetilde{\xi}_{jk}^{n+\frac{1}{2}} |^2 (\widetilde{\xi}_{jk}^{n+\frac{1}{2}} )^2.
\end{aligned}
\end{equation*}
We only provide the estimate of $P_2$ as an example. Using the boundedness of the projection solution to see
\begin{equation*}
\begin{aligned}
				\|P_2\|_{l^2}^2 & \leq C ( \| \widetilde{\xi}^{n+\frac{1}{2}} \|_{l^2}^2 + \| \widetilde{\xi}^{n+\frac{1}{2}} \|_{l^4}^4  ). \\
\end{aligned}
\end{equation*}
By combining Lemma \ref{Poincare}, \ref{Sobolev_embedding2d} and \ref{priori-estimate2d}, we arrive at
\begin{equation*}
				\|\xi^n\|_{l^2} \leq \|U^{\star}(t_n)\|_{l^2} + \|u^n\|_{l^2} \leq C \ \text{and} \ 
				\|\xi^n\|_{l^p}^p \leq C \|\xi^n\|_{l^2}^2, \ \forall p \in [2, \infty).
\end{equation*}
Consequently, 
\begin{equation*}
				\|P_2\|_{l^2}^2 \leq C \| \widetilde{\xi}^{n+\frac{1}{2}} \|_{l^2}^2 \leq C( \|\xi^{n-1}\|_{l^2}^2 + \|\xi^n\|_{l^2}^2 ).
\end{equation*}
The estimates of the remaining terms in \eqref{expandG-G} are analogous. An application of the triangular inequality yields
\begin{equation*}
				\|\widetilde{G}^{\star n+\frac{1}{2}}-\widetilde{G}^{n+\frac{1}{2}}\|^2_{l^2} \leq C (\|\xi^{n-1}\|_{l^2}^2 + \|\xi^n\|_{l^2}^2).
\end{equation*}
Analogously,
\begin{equation*}
				\|\widetilde{g}^{\star n+\frac{1}{2}}-\widetilde{g}^{n+\frac{1}{2}}\|^2_{l^2} \leq C (\|\xi^{n-1}\|_{l^2}^2 + \|\xi^n\|_{l^2}^2).
\end{equation*}
According to \eqref{sigma1-trans} and Lemma \ref{norm-equivalence2d}, \ref{Poincare}, the estimate of $E_1^{n+\frac{1}{2}}$ is thus obtained. The estimate of $E_2^{n+\frac{1}{2}}$ is similar and we omit it here.
\end{proof}

Acting $-\Delta_\mathcal{N}$ on both sides of the first equation in \eqref{error-eq}, then taking the discrete inner product on both sides of the resulting equation with $\xi^{n+1}$, subsequently using the Cauchy-Schwarz inequality and $a^2 + 2ab + b^2 \leq (1 + \tau)a^2 + (1 + \tau^{-1})b^2$ to get
\begin{equation*}
\begin{aligned}
				|\eta^{u,n+1}|_{\mathcal{N},1}^2 &\leq  \ |e^{11}(\tau) \xi^n + e^{12}(\tau) \eta^n|_{\mathcal{N},1}^2 + |\tau \beta  e^{12}\big(\tfrac{\tau}{2}\big) E_2^{n+\frac{1}{2}} - R_1^{n+1} |_{\mathcal{N},1}^2 \\
				&\quad + 2| e^{11}(\tau) \xi^n + e^{12}(\tau) \eta^n|_{\mathcal{N},1} |\tau \beta e^{12}\big(\tfrac{\tau}{2}\big) E_2^{n+\frac{1}{2}} - R_1^{n+1} |_{\mathcal{N},1} \\
				&\leq  (1+\tau)|e^{11}(\tau) \xi^n + e^{12}(\tau) \eta^n|_{\mathcal{N},1}^2 + C \tau \|E_2^{n+\frac{1}{2}}\|_{l^2}^2 + 2(1 + \tau^{-1}) |R_1^{n+1}|_{\mathcal{N},1}^2.
\end{aligned}
\end{equation*}
It can be further simplified as follows according to Lemma \ref{norm-equivalence2d}, \ref{Poincare}, \ref{spectral-radius}
\begin{equation}\label{energy1}
				|\eta^{u,n+1}|_{\mathcal{N},1}^2 \leq (1+\tau) |e^{11}(\tau) \xi^n + e^{12}(\tau) \eta^n|_{\mathcal{N},1}^2 + C \tau \|E_2^{n+\frac{1}{2}}\|_{l^2}^2 + C \tau^{-1} |R_1^{n+1}|_{\mathcal{N},1}^2.
\end{equation}
Taking the discrete inner product on both sides of the second equation in \eqref{error-eq} by $\eta^{n+1}$ and making some calculations, we can analogously derive
\begin{equation}\label{energy2}
				\|\eta^{v,n+1}\|_{l^2}^2  \leq (1+\tau) \|e^{21}(\tau) \xi^n + e^{22}(\tau) \eta^n\|_{l^2}^2 + C \tau \|E_2^{n+\frac{1}{2}}\|_{l^2}^2 +  C\tau^{-1} \|R_2^{n+1}\|_{l^2}^2.
\end{equation}
Multiplying both sides of the last equation of \eqref{error-eq} by $2\tau \zeta^{n+\frac{1}{2}}$ and using the Cauchy-Schwarz inequality yield
\begin{equation}\label{energy3}
				|\eta^{r,n+1}|^2 \leq (1+\tau)(|\zeta^n|^2 + |\zeta^{n+1}|^2) + C \tau |E_1^{n+\frac{1}{2}}|^2 + C \tau|R_3^{n+\frac{1}{2}}|^2.
\end{equation}
Adding \eqref{energy1}, \eqref{energy2} and \eqref{energy3} together and using Lemma \ref{Poincare}, \ref{sigma-l2}, we obtain
\begin{equation} \label{energy4}
\begin{aligned}
				&|\xi^{n+1}|_{\mathcal{N},1}^2 - |\xi^n|_{\mathcal{N},1}^2 + \|\eta^{n+1}\|_{l^2}^2 - \|\eta^n\|_{l^2}^2 + |\zeta^{n+1}|^2 - |\zeta^n|^2 \\
				&\leq C\tau ( |\xi^{n+1}|_{\mathcal{N},1}^2 + |\xi^n|_{\mathcal{N},1}^2 + |\xi^{n-1}|_{\mathcal{N},1}^2 + \|\eta^{n+1}\|_{l^2}^2 + \|\eta^n\|_{l^2}^2 \\
				&+ |\zeta^{n+1}|^2 + |\zeta^n|^2 + \tau^4 + N^{-2m+2}).
\end{aligned}
\end{equation}
It is worth mentioning that the following identity
\begin{equation*}
				|e^{11}(\tau) \xi^n + e^{12}(\tau) \eta^n|_{\mathcal{N},1}^2 + \|e^{21}(\tau) \xi^n + e^{22}(\tau) \eta^n\|_{l^2}^2 = \left\langle e^{\tau A_\mathcal{N}}  (\xi^n, \eta^n)^\top,  e^{\tau A_\mathcal{N}} L_\mathcal{N} (\xi^n, \eta^n)^\top \right\rangle_{l^2},
\end{equation*}
and Lemma \ref{identity} are utilized here. Let $\Phi^k = |\xi^{k}|_{\mathcal{N},1}^2 + \|\eta^k\|_{l^2}^2 + |\zeta^k|^2$, substituting the superscript $n$ with $k$ in \eqref{energy4} and summing over $k$ from $1$ to $n-1$ lead to
\begin{equation}\label{gronwall1}
				\Phi^n - \Phi^1 \leq C\tau \sum\limits_{k=0}^n \Phi^k + C(\tau^4 + N^{-2m+2}).
\end{equation}
For the start-up scheme, we can derive
\begin{equation}\label{gronwall2}
				\Phi^1 - \Phi^0 \leq C\tau(\Phi^0 + \Phi^1) + C(\tau^4 + N^{-2m+2}),
\end{equation}
by a same process. Adding \eqref{gronwall1} and \eqref{gronwall2} together, followed by using the discrete Gronwall-inequality, we obtain 
\begin{equation*}
				\Phi^n \leq C\Phi^0 + C(\tau^4 + N^{-2m+2}). 
\end{equation*}
In practical implementations, $U^0$ and $V^0$ are usually chosen as the $L^2$ projection or spatial interpolation of the initial conditions. Consequently, we have $\Phi^0 \leq CN^{-2m+2}$. The proof is thus completed after using the triangular inequality.

\textbf{II. Convergence for the case $\beta < 0$.} Unlike $\beta>0$, we cannot obtain a priori estimate as there is one non-positive term in the energy expression. Fortunately, since the proposed scheme is linearly implicit, we can use an induction argument to obtain the desired result.

For $n=0$, the numerical solution is usually the $L^2$ projection or the spatial interpolation of the initial conditions, and the conclusion is straightforward from Lemma \ref{approximation-prop-discrete}. For $n=1$, although the scheme is somewhat different from those for $n \geq 2$, the process of proving their convergence is similar, and we omit the proof for $n = 1$ as well. Suppose these results are valid for $n = 2,\cdots,l$, then we prove the estimates as well as the $L^\infty$ boundedness for $n=l+1$. 

Taking the estimate of $E_1^{l+\frac{1}{2}}$ as an example, we demonstrate that Lemma \ref{sigma-l2} is still valid here. From Lemma \ref{norm-equivalence2d}, \ref{Poincare}, \ref{nonlinear-lem}, we have
				\begin{equation}\label{sigma1-trans2}
				\begin{aligned}
								|E_1^{l+\frac{1}{2}}| &\leq C \|\widetilde{f}_\mathcal{N}^{l+\frac{1}{2}}\|_{l^\infty}\| \mathcal{A}(\xi^l, \xi^{l+1}, \eta^l, \eta^{l+1})\|_{l^2}  \\ 
												           &+ C (\|\widetilde{f}_\mathcal{N}^{l+\frac{1}{2}}\|_{l^\infty} \|\widetilde{G}^{\star l+\frac{1}{2}}-\widetilde{G}^{l+\frac{1}{2}}\|_{l^2} + \|\widetilde{g}^{\star l+\frac{1}{2}}-\widetilde{g}^{l+\frac{1}{2}}\|_{l^2})  \\ 
																	 &\leq C (|\xi^l|_{\mathcal{N},1} + |\xi^{l+1}|_{\mathcal{N},1} + \|\eta^l\|_{l^2} + \|\eta^{l+1}\|_{l^2}) \\
																	 &+ C(\|\widetilde{U}^{\star,l+\frac{1}{2}}\|^2_{l^\infty} + \|\widetilde{U}^{\star,l+\frac{1}{2}}\|_{l^\infty} \|\widetilde{u}^{l+\frac{1}{2}}\|_{l^\infty} + \|\widetilde{u}^{l+\frac{1}{2}}\|^2_{l^\infty}) \|\xi^{l+\frac{1}{2}}\|_{l^2} \\ 
																	 &+ C(\|\widetilde{U}^{\star,l+\frac{1}{2}}\|_{l^\infty}^2 + \|\widetilde{u}^{l+\frac{1}{2}}\|_{l^\infty}^2)(\|\widetilde{U}^{\star,l+\frac{1}{2}}\|_{l^\infty} + \|\widetilde{u}^{l+\frac{1}{2}}\|_{l^\infty}) \|\xi^{l+\frac{1}{2}}\|_{l^2} \\ 
																	 &\leq C(|\xi^{l-1}|_{\mathcal{N},1} + |\xi^l|_{\mathcal{N},1} + |\xi^{l+1}|_{\mathcal{N},1} + \|\eta^l\|_{l^2} + \|\eta^{l+1}\|_{l^2} ).
				\end{aligned}
				\end{equation}
In the above proof, we use the $L^\infty$ boundedness of $U^{l-1}$, $U^l$ from the induction instead of using their $L^p$ boundedness as in the proof of Theorem \ref{main-thm}. The estimate of $E_2^{l+1}$ in Lemma \ref{sigma-l2} can be derived analogously. Therefore, by using the same process of proving Theorem \ref{main-thm}, we can still obtain the following estimate
\begin{equation}\label{ul}
				|U^{l+1} - u^{l+1}|_{\mathcal{N},1} + \|V^{l+1} - v^{l+1}\|_{l^2} + |R^{l+1} - r^{l+1}| \leq C(\tau^2 + N^{-m+1}), 
\end{equation}
The difference is that we need to recover the $L^\infty$ boundedness of $u^{l+1}$ here. To this end, we provide an $H^2$ estimate for $u^{l+1}$.

We first indicate $|r^{l+1}| \leq C$. A combination of the third equation of \eqref{fully1} and Lemma \ref{Poincare}, \ref{nonlinear-lem} leads to 
\begin{equation}\label{bound-R}
\begin{aligned}
				|r^{l+1}| &\leq |r^l| +  C \tau  \|\widetilde{f}_\mathcal{N}^{l+\frac{1}{2}}\|_{l^2} (|u^l|_{\mathcal{N},1} + |u^{l+1}|_{\mathcal{N},1} + \|v^l\|_{l^2} + \|v^{l+1}\|_{l^2}) \\ 
								  &\leq |r^l| +  C \tau (\|u^{l-1}\|_{l^\infty} + \|u^l\|_{l^\infty}) (|u^l|_{\mathcal{N},1} + |u^{l+1}|_{\mathcal{N},1} + \|v^l\|_{l^2} + \|v^{l+1}\|_{l^2}).
\end{aligned}
\end{equation}
From the estimates of $U^{l+1}$ and $V^{l+1}$ \eqref{ul} and the induction, we let $\tau$ sufficiently small and $N$ sufficiently large, such that
\begin{equation*}
\begin{aligned}
				&|u^{l+1}|_{\mathcal{N},1} \leq |U^{l+1} - u^{l+1}|_{\mathcal{N},1} + |U^{l+1}|_{\mathcal{N},1} \leq M_1+1, \\
				&\|v^{l+1}\|_{l^2} \leq \|V^{l+1} - v^{l+1}\|_{l^2} + \|V^{l+1}\|_{l^2} \leq M_3+1.
\end{aligned}
\end{equation*}
Thus, we obtain the boundedness of $r^{l+1}$ from \eqref{bound-R}. Combining the boundedness of $r^{l+1}$ and the $L^\infty$ boundedness of $u^n (n \leq l)$, we can get the following estimate for $|E_2^{l+\frac{1}{2}}|_{\mathcal{N},1}$.
\begin{lem}\label{sigma-h1} $|E_2^{l+1}|_{\mathcal{N},1}$ is bounded in the sense
\begin{equation*} 
				|E_2^{l+\frac{1}{2}}|_{\mathcal{N},1}^2 \leq C ( |\xi^{l-1}|_{\mathcal{N},2}^2 + |\xi^l|_{\mathcal{N},2}^2 + |\zeta^l|^2 + |\zeta^{l+1}|^2 ).
\end{equation*}
\end{lem}
\noindent We leave the proof of Lemma \ref{sigma-h1} to Appendix \ref{C}. 

Acting both sides of the first equation in \eqref{error-eq} by $\Delta_\mathcal{N}^2$, then taking the discrete inner product on both sides of the resulting equation with $\xi^{l+1}$, subsequently using the Cauchy-Schwarz inequality, Lemma \ref{Poincare}, \ref{spectral-radius}, we get
 \begin{equation}\label{energy21}
				|\xi^{l+1}|_{\mathcal{N},2}^2 \leq (1+\tau) |e^{11}(\tau) \xi^l + e^{12}(\tau) \eta^l|_{\mathcal{N},2}^2 + C \tau |E_2^{l+\frac{1}{2}}|_{\mathcal{N},1}^2 + C \tau^{-1} |R_1^{l+1}|_{\mathcal{N},2}^2.
\end{equation}
Acting both sides of the second equation in \eqref{error-eq} by $-\Delta_\mathcal{N}$ and taking the discrete inner product on both sides of the resulting equation by $\eta^{l+1}$ give
\begin{equation}\label{energy22}
				|\eta^{l+1}|_{\mathcal{N},1}^2  \leq (1+\tau) |e^{21}(\tau) \xi^l + e^{22}(\tau) \eta^l|_{\mathcal{N},1}^2 + C \tau |E_2^{l+\frac{1}{2}}|_{\mathcal{N},1}^2 + C \tau^{-1} |R_2^{l+1}|_{\mathcal{N},1}^2.
\end{equation}
The derivations of \eqref{energy21} and \eqref{energy22} are similar to those of \eqref{energy1} and \eqref{energy2}. Furthermore, \eqref{energy3} can also be proved here by using the same process. Adding \eqref{energy21}, \eqref{energy22} and \eqref{energy3} together and using Lemma \ref{identity}, \ref{Poincare}, \ref{local-error}, \ref{sigma-l2}, we obtain
\begin{equation} \label{energy24}
\begin{aligned}
				& |\xi^{l+1}|_{\mathcal{N},2}^2 - |\xi^l|_{\mathcal{N},2}^2 + |\eta^{l+1}|_{\mathcal{N},1}^2 - |\eta^l|_{\mathcal{N},1}^2 + |\zeta^{l+1}|^2 - |\zeta^l|^2 \\
				&\leq C\tau ( |\xi^{l+1}|_{\mathcal{N},2}^2 + |\xi^l|_{\mathcal{N},2}^2 + |\xi^{l-1}|_{\mathcal{N},2}^2 + |\eta^{l+1}|_{\mathcal{N},1}^2 + |\eta^l|_{\mathcal{N},1}^2\\ 
				&+ |\zeta^{l+1}|^2 + |\zeta^l|^2 + \tau^4 + N^{-2m+4}).
\end{aligned}
\end{equation}
Let $\Psi^k = |\xi^{k}|_{\mathcal{N},2}^2 + |\eta^{k}|_{\mathcal{N},1}^2 + |\zeta^k|^2$, then substituting the superscript $l$ with $k$ in \eqref{energy24} and summing up $k$ from $1$ to $l$ yield
\begin{equation*}\label{gronwall21}
				\Psi^{l+1} - \Psi^1 \leq C\tau \sum\limits_{k=0}^l \Psi^k + C(\tau^4 + N^{-2m+4}).
\end{equation*}
An application of the discrete Gronwall inequality and the induction leads to
\begin{equation*}
				\Psi^{l+1} \leq C\Psi^1 + C(\tau^4 + N^{-2m+4}) \leq C (\tau^4 + N^{-2m+4}),
\end{equation*}
which implies
\begin{equation*}
				|\xi^{l+1}|_{\mathcal{N},2} \leq C (\tau^2 + N^{-m+2}).
\end{equation*}
Consequently, by using Lemma \ref{norm-equivalence2d}, \ref{Sobolev_embedding2d_Linfty}, we can let $N>N_0$ sufficient large and $\tau < \tau_0$ sufficiently small such that
\begin{equation*}
\begin{aligned}
				\|u^{l+1}\|_{l^\infty} &\leq \|U^{\star}(t_{l+1})\|_{l^\infty} + \|\xi^{l+1}\|_{l^\infty} \leq \|U^{\star}(t_{l+1})\|_{l^\infty} + C|\xi^{l+1}|_{\mathcal{N},2} \\ 
															 &\leq M_1 + C(\tau^4 + N^{-m+2}) \leq M_1 + 1.
\end{aligned}
\end{equation*}
Note that the last inequality requires the hypothesis $m \geq 3$. Therefore, the result is true for $n = l + 1$. We thus finish the whole convergence analysis.
\begin{coro}
				The linear system of the \textbf{SAV-IF} scheme \eqref{fully1} is uniquely solvable.
\end{coro}
\begin{proof}
				From the implementation of the scheme, we only need to demonstrate that the denominator of \eqref{update-r} is nonzero. To this end, we next prove that $b_2^{n+\frac{1}{2}}$ can be bounded uniformly when $\tau$ is sufficiently small and $N$ is sufficiently large. We get by combining its definition, the Cauchy-Schwarz and the triangular inequality that
				\begin{equation*}
				\begin{aligned}
								|b_2^{n+\frac{1}{2}}|^2 &\leq \frac{1}{2} \|f_\mathcal{N}(\widetilde{u}^{n+\frac{1}{2}})\|^2 (|u_2^{n+1}|_{\mathcal{N},1} + \|v_2^{n+1}\|_{l^2})^2 \\
																			&\leq 2 \|f_\mathcal{N}(\widetilde{u}^{n+\frac{1}{2}})\|(|u_1^{n+1}|_{\mathcal{N},1} + |u^{n+1}|_{\mathcal{N},1} + \|v_1^{n+1}\|_{l^2} + \|v^{n+1}\|_{l^2}) \\
																			&\leq 2 \|f_\mathcal{N}(\widetilde{u}^{n+\frac{1}{2}})\| (|u^n|_{\mathcal{N},1} + |u^{n+1}|_{\mathcal{N},1} + \|v^n\|_{l^2} + \|v^{n+1}\|_{l^2}).
				\end{aligned}
				\end{equation*}
				The last inequality is due to the identity
				\begin{equation*}
								|u_1^{n+1}|_{\mathcal{N},1}^2 + \|v_1^{n+1}\|_{l^2}^2 = |u^n|_{\mathcal{N},1}^2 + \|v^n\|_{l^2}^2.
				\end{equation*}
				By using the convergence result, the uniform boundedness of $b_2^{n+\frac{1}{2}}$ can be obtained straightforward and the proof is thus completed.
\end{proof}
\section{Arbitrary high-order linear energy-preserving schemes}\label{sec.4}
The SAV reformulation \eqref{sav} can also provide an elegant platform for developing high-order schemes. In this section, we propose a framework for constructing arbitrarily high-order linear energy-preserving algorithms based on the integrating factor Runge-Kutta methods \cite{exponential-intergrators}. Supposing that $u^n$ and $v^n$ have been precomputed, we employ the Runge-Kutta methods and the extrapolation technique to \eqref{Lawson-system}, then rewrite the obtained system in terms of the original variables by using the discrete Lawson transform. The resulting \textbf{SAV-IFRK} methods are described as follows:
\begin{alg}[\textbf{SAV-IFRK} method]
Computing $u^{ni}, v^{ni}$ and $r^{ni}$ from
\begin{equation}\label{SAV-IFRK-inter}
				\left\lbrace
\begin{aligned}
				u^{ni} &= e^{11}(c_i \tau) u^n + e^{12}(c_i \tau) v^n - \tau \beta \sum\limits_{j=1}^s a_{ij}  e^{12}((c_i - c_j)\tau)r^{nj} f_\mathcal{N} (u^{nj}), \\ 
				v^{ni} &= e^{21}(c_i \tau) u^n + e^{22}(c_i \tau) v^n - \tau \beta \sum\limits_{j=1}^s a_{ij}  e^{22}((c_i - c_j)\tau)r^{nj} f_\mathcal{N} (u^{nj}), \\  
				r^{ni} &= r^n + \tau \sum\limits_{j=1}^s a_{ij} (f_\mathcal{N}(\widetilde{u}^{nj}), v^{nj})_{l^2}.  \\ 
\end{aligned}
\right.
\end{equation}
Then, the numerical solutions $u^{n+1}, v^{n+1}$ and $r^{n+1}$ are updated by
\begin{equation}\label{SAV-IFRK-update}
				\left\lbrace
\begin{aligned}
				u^{n+1} &= e^{11}(\tau) u^n + e^{12}(\tau) v^n - \tau \beta \sum\limits_{i=1}^s b_i  e^{12}((1 - c_i)\tau) r^{ni} f_\mathcal{N} (u^{ni}), \\
				v^{n+1} &= e^{21}(\tau) u^n + e^{22}(\tau) v^n - \tau \beta \sum\limits_{i=1}^s b_i  e^{22}((1 - c_i)\tau) r^{ni} f_\mathcal{N} (u^{ni}), \\
				r^{n+1} & = r^n + \tau\sum\limits_{i=1}^s b_i (f_\mathcal{N}(\widetilde{u}^{ni}), v^{ni})_{l^2},
\end{aligned}
				\right.
\end{equation}
where $\widetilde{u}^{ni}$ is any suitable approximation of $u_\mathcal{N}(t_n + c_i\tau)$. 
\end{alg}
In this paper, we will employ an iterative procedure to obtain sufficiently accurate predictions of $u(t_n + c_i \tau)$ without loss of accuracy as well as efficiency. The specific process is described as follows.

Let $u^{ni,0} = u^n$ and $M>0$  be a positive integer. For $m=0$ to $M-1$, we compute $u^{ni,m+1}$ as follows:
\begin{equation*}
				u^{ni,m+1} = e^{11}(c_i \tau) u^n + e^{12}(c_i \tau) v^n - \tau \sum\limits_{j=1}^s a_{ij}e^{12}((c_i - c_j)\tau) g(u^{nj,m}),
\end{equation*}
if $\max\limits_{i} \| u^{ni,m+1} - u^{ni,m}\|_{l^\infty} < TOL$, we stop the iteration and set $\widetilde{u}^{ni} = u^{ni,m+1}$; otherwise, we set $\widetilde{u}^{ni} = u^{ni,M}$. Then, we proceed to update $u^{n+1}$ by \eqref{SAV-IFRK-inter} and \eqref{SAV-IFRK-update}.
\begin{lem}\label{ifrk-e-lem} 
				For any linear operator $L$, we denote $L^H$ its adjoint, i.e.,
\begin{equation*}
				(L^H v, w)_{l^2} = (v, Lw)_{l^2} \ \forall \ v,w \in \mathcal{M}_\mathcal{N}, \quad  \text{or} \quad \langle L^H \bm{f}, \bm{g} \rangle_{l^2} := \langle \bm{f}, L\bm{g} \rangle_{l^2}, \quad \forall \ \bm{f}, \bm{g} \in \mathcal{M}_\mathcal{N} \times \mathcal{M}_\mathcal{N}.
\end{equation*}
Then, the following identities hold:
\begin{enumerate}
				\item $\mathcal{S}^H = \mathcal{S}^{-1}$, 
				\item $(e^{\mu\nu}(\tau))^H = \overline{ e^{\mu \nu}(\tau) }, \ \forall \mu, \nu = 1, 2, \ e^{21}(\tau) = \Delta_\mathcal{N} e^{12}(\tau)$,
				\item $\overline{e^{11}(-\tau)} = e^{11}(\tau), \ \overline{e^{12}(-\tau)} = -e^{12}(\tau), \ \overline{e^{21}(-\tau)} = -e^{21}(\tau), \ \overline{e^{22}(-\tau)} = e^{22}(\tau)$,  
				\item $L_\mathcal{N} e^{\tau A_\mathcal{N}} = e^{-\tau A_\mathcal{N}^H} L_\mathcal{N} = (e^{-\tau A_\mathcal{N}})^H L_\mathcal{N}$ .
\end{enumerate}
\end{lem}
\begin{proof}
				The first equation can be obtained by the orthogonality of sine basis \cite{GP-SPS}. Combining the first identity,
				the Proposition \ref{exp-element2d} and some calculations, we can arrive at the second and the third equations. We prove the last equation accordingly. It is readily to get $(e^{\tau A_\mathcal{N}})^H = e^{\tau A_\mathcal{N}^H}$ . For any $\bm{f}, \bm{g} \in \mathcal{M}_\mathcal{N} \times \mathcal{M}_\mathcal{N}$, we have
				\begin{equation*}
				\begin{aligned}
								\langle e^{-\tau A_\mathcal{N}^H} L_\mathcal{N} \bm{f}, \bm{g} \rangle_{l^2} &= (- \overline{e^{11}(-\tau)} \Delta_\mathcal{N} f_1,  g_2)_{l^2} + (-\overline{e^{12}(-\tau)}\Delta_\mathcal{N} f_1, g_2)_{l^2}  \\ 
				&+ (\overline{e_{21}(-\tau)} f_2, g_1)_{l^2} + (\overline{e^{21}(-\tau)} f_2, g_1)_{l^2} \\ 
								&= (-\Delta_\mathcal{N} e_{11}(\tau) f_1,  g_2)_{l^2} + (e^{21}(\tau) f_1, g_2)_{l^2} \\ 
								&+ (-\Delta_\mathcal{N} e^{12}(\tau) f_2, g_1)_{l^2} + (e^{22}(\tau) f_2, g_2)_{l^2} = \langle L_\mathcal{N} e^{\tau A_\mathcal{N}} \bm{f}, \bm{g} \rangle_{l^2}.
				\end{aligned}
				\end{equation*}
Consequently, $e^{-\tau A_\mathcal{N}^H} L_\mathcal{N} = L_\mathcal{N} e^{\tau A_\mathcal{N}}$, and the proof is thus completed.
\end{proof}
\begin{thm}\label{conservation-ifrk}
				The \textbf{SAV-IFRK} approach preserves the following fully discrete energy
				\begin{equation*}
								E_\mathcal{N}^{M} = E_\mathcal{N}^{M-1} = \cdots = E_\mathcal{N}^0, \quad E_\mathcal{N}^n = |u^n|_{\mathcal{N},1}^2 + \|v^n\|_{l^2}^2 + \beta (r^n)^2 - \beta C_0,
				\end{equation*}
				if the coefficients $a_{ij}, b_i$ satisfy the conditions
				\begin{equation}\label{RK-coef}
								b_i a_{ij} + b_{j} a_{ji} = b_i b_j , \ \forall \ i,j =1,\cdots s.
				\end{equation}
\end{thm}
\begin{proof}
				Let $F_\mathcal{N}^{ni} = \left(\begin{smallmatrix} 0 \\ \beta r^{ni} f_\mathcal{N} (u^{ni}) \end{smallmatrix}\right)$. Then, the first two equations in \eqref{SAV-IFRK-inter} and \eqref{SAV-IFRK-update} can be rewritten more compactly into 
				\begin{equation}\label{compact}
								\bm{z}^{ni} = e^{c_i \tau A_\mathcal{N}} \bm{z}^n - \tau \sum\limits_{j=1}^s a_{ij} e^{(c_i - c_j)\tau A_\mathcal{N}} F_\mathcal{N}^{ni}, \quad
								\bm{z}^{n+1} = e^{\tau A_\mathcal{N}} \bm{z}^n - \tau \sum\limits_{i=1}^s b_i e^{(c_i - c_j)\tau A_\mathcal{N}} F_\mathcal{N}^{ni}.
				\end{equation}
Taking the discrete inner product on both sides of the second equation with $L_\mathcal{N} \bm{z}^{n+1}$, and subtracting $\left\langle \bm{z}^n, L_\mathcal{N} \bm{z}^n \right\rangle_{l^2}$ from both sides of the resulting equation, we get
\begin{equation}\label{ifrk-e1}
				\langle \bm{z}^{n+1}, L_\mathcal{N} \bm{z}^{n+1} \rangle_{l^2} - \langle \bm{z}^n, L_\mathcal{N}\bm{z}^n \rangle_\mathcal{N} = \varpi_1 + \varpi_2.
\end{equation}
Combining the first equation of \eqref{compact}, Lemma \ref{ifrk-e-lem} and the condition \eqref{RK-coef}, it gives
\begin{equation}\label{ifrk-e2}
\begin{aligned}
				\varpi_1 &= -2\tau  \sum\limits_{i=1}^s b_i \Re \langle e^{\tau A_\mathcal{N}} \bm{z}^n, L_\mathcal{N} e^{(1-c_i)\tau A_\mathcal{N}} F_\mathcal{N}^{ni}\rangle_{l^2} = -2\tau \sum\limits_{i=1}^s b_i \Re \langle e^{c_i \tau A_\mathcal{N}} \bm{z}^n, L_\mathcal{N} F_\mathcal{N}^{ni}\rangle_{l^2}, \\
				\varpi_2 &= \tau^2 \sum\limits_{i,j = 1}^s b_i b_j \langle F_\mathcal{N}^{ni}, L_\mathcal{N} e^{(c_i - c_j)\tau A_\mathcal{N}} F_\mathcal{N}^{nj} \rangle_{l^2} = 2\tau^2 \sum\limits_{i,j=1}^s b_i a_{ij} \Re \langle F_\mathcal{N}^{ni}, L_\mathcal{N} e^{(c_i - c_j)\tau A_\mathcal{N}} F_\mathcal{N}^{nj} \rangle_{l^2} \\
								&= 2\tau \sum\limits_{i=1}^s b_i \Re \langle e^{c_i \tau A_\mathcal{N}} \bm{z}^n, L_\mathcal{N} F_\mathcal{N}^{ni}\rangle_{l^2} - 2\tau \sum\limits_{i=1}^s b_i \Re \langle \bm{z}^{ni}, L_\mathcal{N} F_\mathcal{N}^{ni} \rangle_{l^2}.
\end{aligned}
\end{equation}
Multiplying both sides of the fourth equation in \eqref{SAV-IFRK-inter} by $b_i R^{ni}$, and summing up the subscript $i$ yield
\begin{equation}\label{ifrk-e3}
\begin{aligned}
			\sum\limits_{i=1}^s b_i r^{ni} (f_\mathcal{N}(\widetilde{u}^{ni}), v^{ni})_{l^2}  = \sum\limits_{i=1}^s b_i \Re \langle \bm{z}^{ni}, L_\mathcal{N}F_\mathcal{N}^{ni} \rangle_{l^2}.
\end{aligned}
\end{equation}
According to the third equation in \eqref{SAV-IFRK-inter}, the last equation in \eqref{SAV-IFRK-update} and the condition \eqref{RK-coef}, we can analogously derive that
\begin{equation}\label{ifrk-e4}
				(r^{n+1})^2 - (r^n)^2 = 2\tau \sum\limits_{i=1}^s  b_i r^{ni} (f_\mathcal{N}(\widetilde{u}^{ni}), v^{ni})_{l^2}.
\end{equation}
Plugging \eqref{ifrk-e4} into \eqref{ifrk-e3}, and \eqref{ifrk-e2} into \eqref{ifrk-e1}, and adding the resulting equations together yield the final result.
\end{proof}
\begin{rmk}
				We note that s-stage Gauss collocation methods \cite{hairer2006} satisfy the condition \eqref{RK-coef} with order $2s$. In this paper, we provide examples of the fourth-order and sixth-order Gauss methods, and their butcher tables are provided below:
\begin{equation*}
\begin{array}{c | c c}
				\frac{1}{2} - \frac{\sqrt{3}}{6} & \frac{1}{4} & \frac{1}{4} - \frac{\sqrt{3}}{6} \\ 
				\frac{1}{2} + \frac{\sqrt{3}}{6} & \frac{1}{4} + \frac{\sqrt{3}}{6} & \frac{1}{4} \\
				\hline 
																				 & \frac{1}{2} & \frac{1}{2}
\end{array}
\quad 
\begin{array}{c | c c c}
				\frac{1}{2} - \frac{\sqrt{15}}{10} & \frac{5}{36} & \frac{2}{9} - \frac{\sqrt{15}}{15} & \frac{5}{36} - \frac{\sqrt{15}}{30} \\ 
				\frac{1}{2} & \frac{5}{36} + \frac{\sqrt{15}}{24} & \frac{2}{9} & \frac{5}{36} - \frac{\sqrt{15}}{24}  \\
				\frac{1}{2} + \frac{\sqrt{15}}{10} & \frac{5}{36} + \frac{\sqrt{15}}{30} & \frac{2}{9}+\frac{\sqrt{15}}{15} & \frac{5}{36} \\
				\hline 
																					 & \frac{5}{18} & \frac{4}{9} & \frac{5}{18}
\end{array}
.
\end{equation*}
\end{rmk}
We will focus on the fourth-order and sixth-order SAV Gauss Runge-Kutta (\textbf{SAV-IFGRK4}, \textbf{SAV-IFGRK6}) methods in this paper.

\section{Numerical experiments}\label{sec.5}
In this section,  we display the numerical performance of the proposed methods in terms of accuracy, computational efficiency and invariant preservation for simulating the NLSW \eqref{eq1}. In addition to the methods given above, i.e., the \textbf{SAV-IF}, the \textbf{SAV-IFGRK4} and the \textbf{SAV-IFGRK6} methods. We also introduce two other popular methods for comparisons. The first approach proposed in \cite{Wu16SIAM} combines the exponential time difference technique and the discrete gradient methods (\textbf{ETD-DG}). The second approach is the classical SAV Crank-Nicolson method based on the extrapolation technique (\textbf{SAV-CN}) in \cite{SAV-Shen}. Notably, \textbf{ETD-DG} preserves the original Hamiltonian energy while \textbf{SAV-CN} is linearly implicit. Both of them have been widely used recently.  

It should be emphasized that all the above methods will be tested using the sine-pseudo spectral method for spatial discretization. For the fully implicit \textbf{ETD-DG} scheme, we employ the fixed-point iterations given in \cite{Wu16SIAM} to solve the nonlinear system. The remaining SAV-based methods will use implementations similar to that presented in Section \ref{integrator}. The \texttt{dst} algorithm are used in all experiments to accelerate the matrix-vector product.
\begin{ex}
				To investigate the spatial and temporal accuracy of the proposed methods, we provide the following manufactured analytic solution to the NLSW 
				\begin{equation*}
								u(x, y, t) = {\rm sech}(x^2 + y^2) \exp{(-\sqrt{2}\pi i t)}, 
				\end{equation*}
				which can be constructed by adding a nonhomogeneous source term to the RHS of \eqref{eq1}. The initial conditions in \eqref{eq1} are $u_0 = u(x, y, 0)$ and $u_1=\partial_t u(x, y, 0)$, respectively.  The computational domain is set to $\Omega = (-8, 8)$. The parameter $\alpha$ is chosen as $\alpha = 1$.  We will test the convergence rate for both $\beta = 1$ and $\beta = -1$.
\end{ex}

\begin{figure}[H]
				\centering
\subfigure[Error vs $N$]{\label{spatial2d}
				\includegraphics[width = 0.31\linewidth]{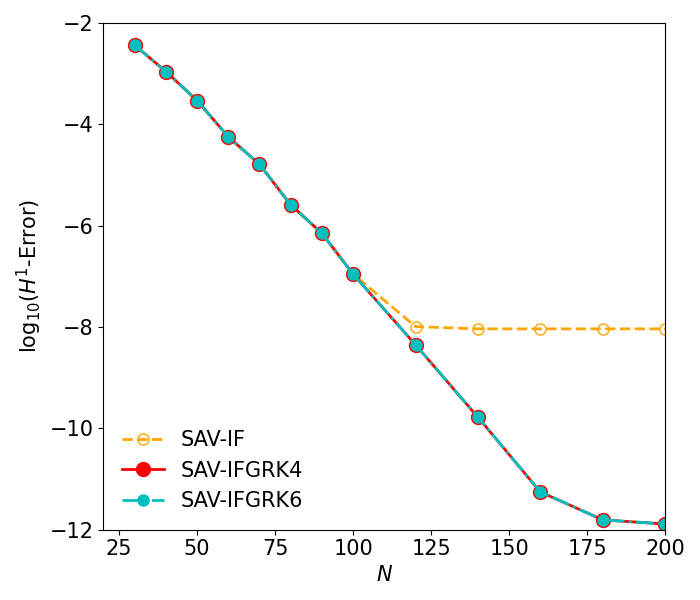}
}
\subfigure[Errors vs $\tau$]{\label{temporal2d}
				\includegraphics[width = 0.31\textwidth]{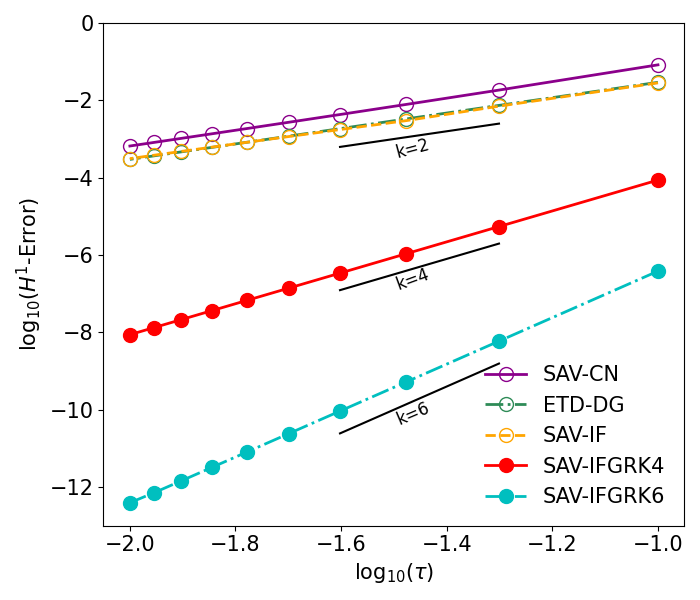}
}
\subfigure[Errors vs CPU]{\label{cpu2d}
				\includegraphics[width = 0.31\textwidth]{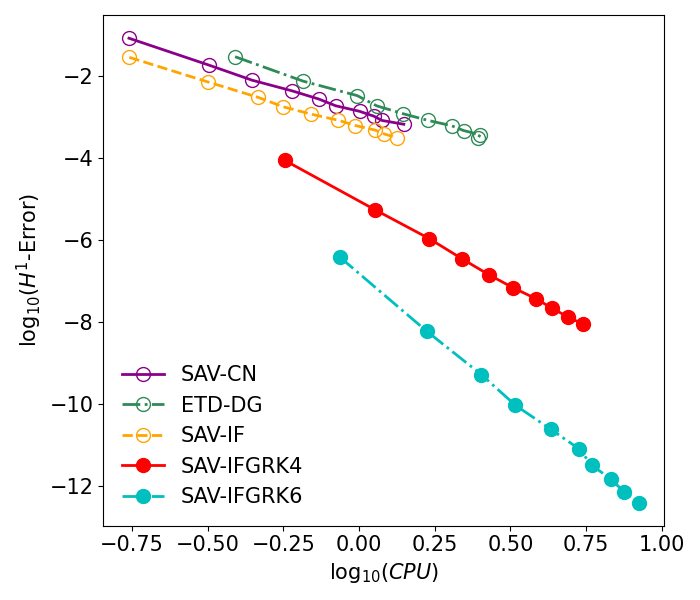}
}
\caption{  \label{fig-ex2-err}
 Spatial accuracy \subref{spatial2d}, temporal accuracy \subref{temporal2d} and efficiency curves of different schemes for the 2D NLSW with $\beta = 1$.}
\end{figure}

\begin{figure}[H]
				\centering
\subfigure[Error vs $N$]{\label{spatial2dm1}
				\includegraphics[width = 0.31\linewidth]{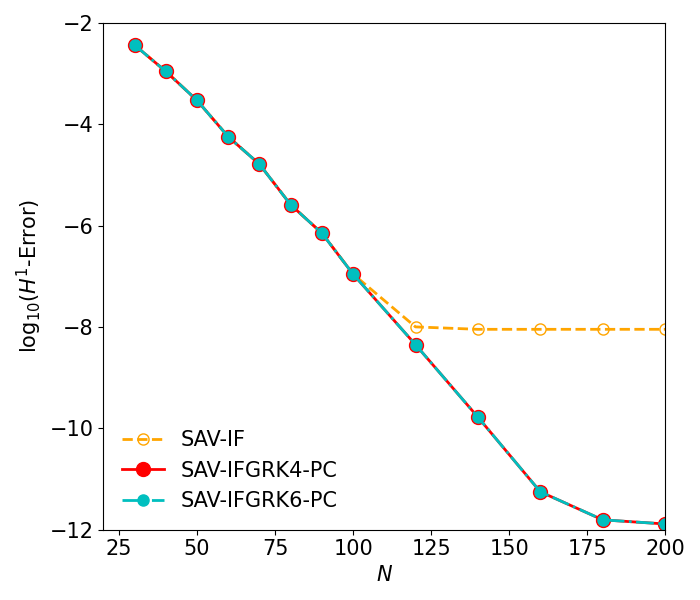}
}
\subfigure[Errors vs $\tau$]{\label{temporal2dm1}
				\includegraphics[width = 0.31\textwidth]{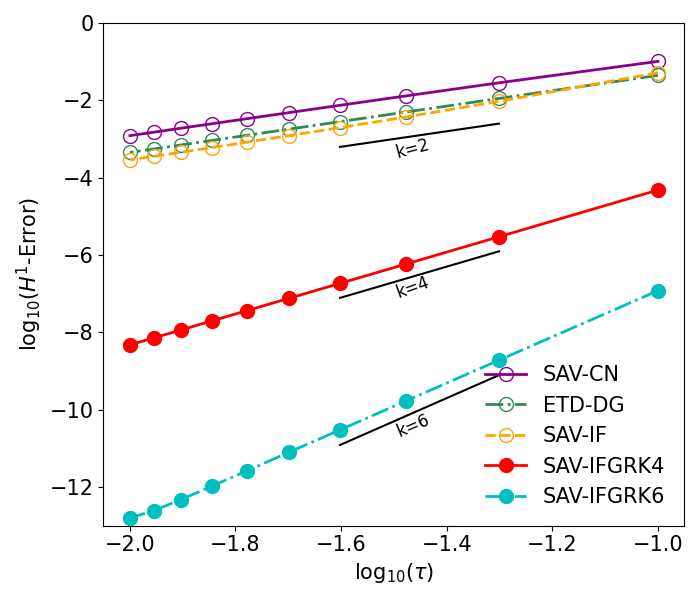}
}
\subfigure[Errors vs CPU]{\label{cpu2dm1}
				\includegraphics[width = 0.31\textwidth]{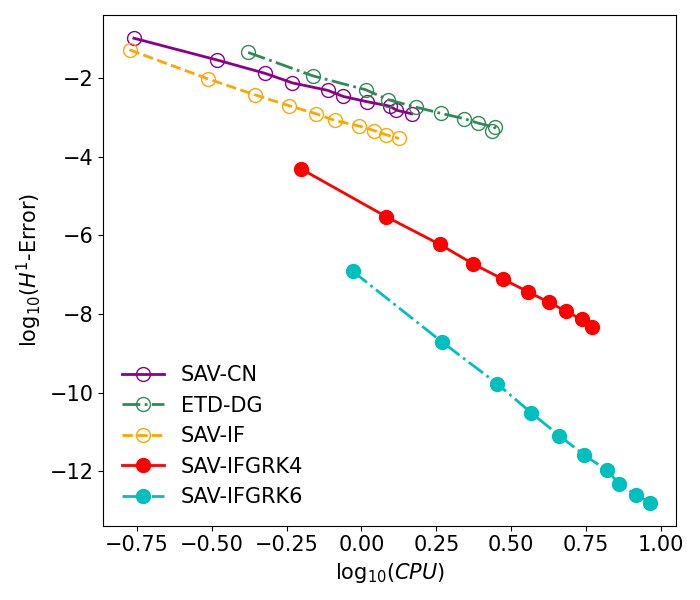}
}
\caption{  \label{fig-ex2-errm1}
 Spatial accuracy \subref{spatial2d}, temporal accuracy \subref{temporal2d} and efficiency curves of different schemes for the 2D NLSW with $\beta = -1$.}
\end{figure}

In the spatial convergence test, the termination time is fixed at $t = 0.1$. We employ the \textbf{SAV-IF}, \textbf{SAV-IFGRK4} and \textbf{SAV-IFGRK6} methods to integrate the NSLW with $\tau = 10^{-5}$ to ignore the error caused by temporal discretization. The mesh refinement test is then performed by varying $N$ from $20$ up to $200$. Figure \ref{spatial2d} and Figure \ref{spatial2dm1} display the logarithm of the solution errors solved by the three methods at $t = 0.1$ in the discrete $H^1$ norm as a function of $N$, respectively with $\beta = 1$ and $\beta = -1$. It is obvious that within a specific range of $N$, the errors decrease exponentially when increasing $N$, displaying a spectral accuracy in space. When $N$ becomes very large, the error curves level off as $N$ further increases, especially for the \textbf{SAV-IF} scheme, showing a situation caused by the temporal error.  

In the temporal convergence test, we fix the integration time at $t = 1$ and let $N = 200$ to make spatial error negligible. Then, we carry out mesh refinement test by varying $\tau = 0.1 \times 10^{-k}$ with $k$ ranging form $1$ to $10$. The discrete $H^1$ errors of numerical solutions at $t=1$ versus the time step with $\beta = 1$ and $\beta = -1$ are respectively reported in Figure \ref{temporal2d} and Figure \ref{temporal2dm1} in logarithmic scales. It is evident that numerical errors of \textbf{SAV-CN}, \textbf{ETD-DG} and \textbf{SAV-IF} methods exhibit a second-order convergence rate in time. For \textbf{SAV-IFGRK4} and \textbf{SAV-IFGRK6}, numerical errors decrease rapidly with forth-order and sixth-order accuracy, respectively. Due to the exact integration of the linear part, although \textbf{SAV-IF}, \textbf{ETD-DG} and \textbf{SAV-CN} all have second-order accuracy, the numerical errors of \textbf{SAV-IF} and \textbf{ETD-DG} are significantly smaller than those of \textbf{SAV-CN}. The above experiments confirm the proposed error estimates numerically. Since the original system will not have conservation properties after adding the source term, we omit the energy diagram in this case.

Besides the accuracy, we also compare the efficiency of the five schemes. We plot the logarithm of the numerical error versus CPU times in Figure \ref{cpu2d} and Figure \ref{cpu2dm1}. We can draw the following conclusions: (i) For \textbf{SAV-IFGRK4} and \textbf{SAV-IFGRK6}, although they require more time for calculating the numerical solution in each time step in comparison with other second-order methods with the same step size. However, their high-order accuracy allows them to obtain more accurate numerical solutions even for large time steps and makes them more efficient than second-order schemes. (ii) While fixing temporal and spatial steps, the numerical errors of \textbf{SAV-IF} and \textbf{ETD-DG} are close. However, due to its explicit implementation, the \textbf{SAV-IF} is much cheaper in terms of updating the solutions. Meanwhile, the  \textbf{SAV-IF} is much more accurate than the \textbf{SAV-CN}. All these factors make \textbf{SAV-IF} the most efficient among three second-order schemes.

\begin{ex}
				Let $\alpha = \beta = 1$ in \eqref{eq1}, we use an example to verify the discrete conservation laws of the proposed methods. The initial conditions are
				\begin{equation*}
								u_0(x,y) = (1+i)(x+y) \exp{(-10(1-x-y)^2)}, \ u_1(x, y) = 0.
				\end{equation*}
We let the spatial domain $\Omega = (-32, 32)^2$ and the final time $t = 10$.
\end{ex}

We let $N = 1280$ for spatial discretization, then integrate the system until $t = 10$ by using different methods.

\begin{figure}[H]
\centering
\begin{minipage}[t]{0.32\textwidth}
				\includegraphics[width = 1\linewidth]{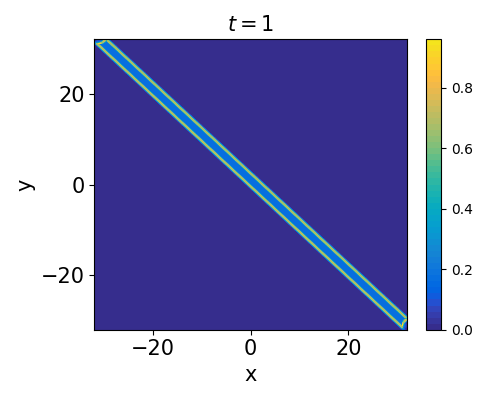}
\end{minipage}
\begin{minipage}[t]{0.32\textwidth}
				\includegraphics[width = 1\linewidth]{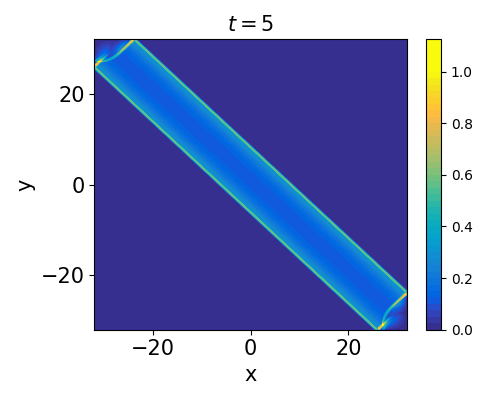}
\end{minipage}
\begin{minipage}[t]{0.32\textwidth}
				\includegraphics[width = 1\linewidth]{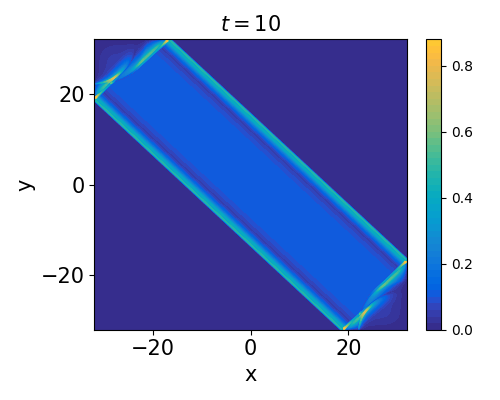}
\end{minipage}
\caption{ \label{snapshot} Snapshots of $|u|$ solved by \textbf{SAV-IFGRK6} at different times.}
\end{figure}
\begin{figure}[H]
\centering
\begin{minipage}[t]{0.24\textwidth}
				\includegraphics[width=1\linewidth]{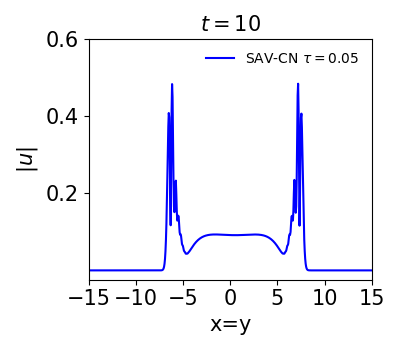}
\end{minipage}
\begin{minipage}[t]{0.24\textwidth}
				\includegraphics[width=1\linewidth]{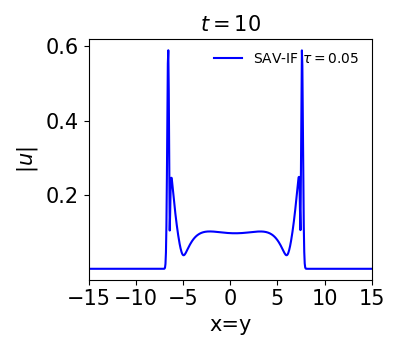}
\end{minipage}
\begin{minipage}[t]{0.24\textwidth}
				\includegraphics[width=1\linewidth]{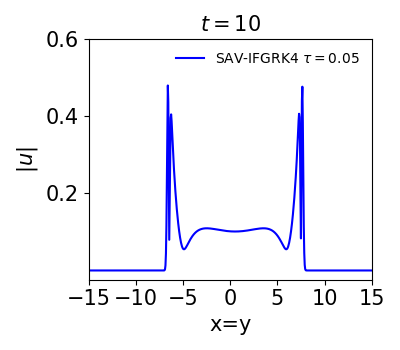}
\end{minipage}
\begin{minipage}[t]{0.24\textwidth}
				\includegraphics[width=1\linewidth]{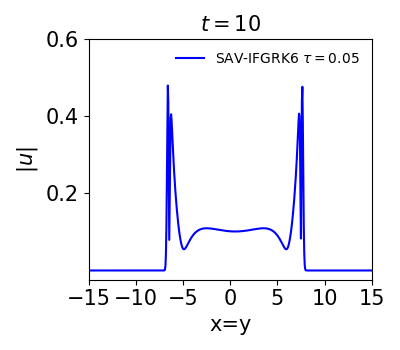}
\end{minipage}
\begin{minipage}[t]{0.24\textwidth}
				\includegraphics[width=1\linewidth]{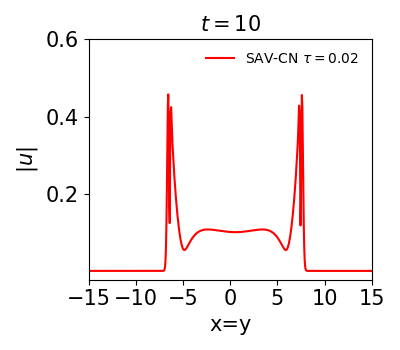}
\end{minipage}
\begin{minipage}[t]{0.24\textwidth}
				\includegraphics[width=1\linewidth]{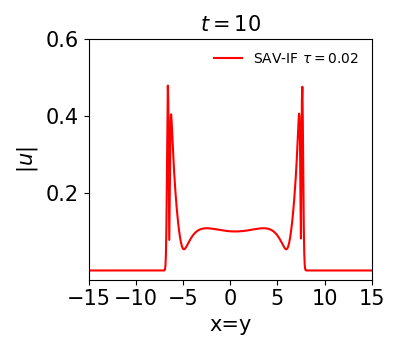}
\end{minipage}
\begin{minipage}[t]{0.24\textwidth}
				\includegraphics[width=1\linewidth]{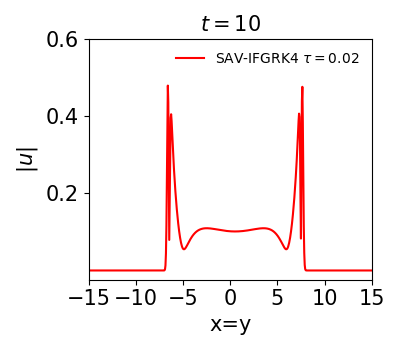}
\end{minipage}
\begin{minipage}[t]{0.24\textwidth}
				\includegraphics[width=1\linewidth]{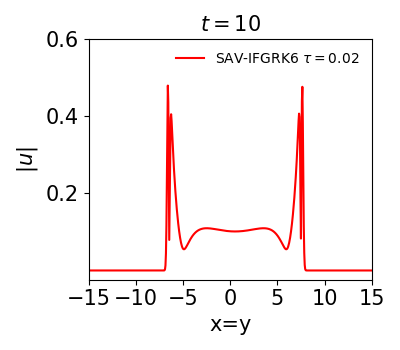}
\end{minipage}
\caption{\label{section view} Section views of $|u|$ solved by four methods with different time steps at $t=10$. }
\end{figure}

Let $\tau = 0.05$, we display the snapshots of $|u|$ solved by \textbf{SAV-IFRK6} at $t = 1, 5, 10$ in Figure \ref{snapshot}. Our simulation results are well agreed with those obtained in \cite{SAV-NLSW} . For detailed comparisons, we give the section views of $|u|$ with $x = y$ in the domain $[-15, 15]$ at $t=10$ solved by four methods. The first row and second row of Figure \ref{section view} displays the simulation results under the time step $\tau = 0.05$ and $\tau = 0.02$, respectively. One can observe that the \textbf{SAV-IFGRK4} and \textbf{SAV-IFGRK6} can capture the oscillatory waves well even under a relatively large time step. Due to the limitation of accuracy, the oscillatory waves failed to be characterized perfectly by \textbf{SAV-IF} under the time step $\tau = 0.05$, but is well captured as we refine the time grid to $\tau = 0.02$. It is evident that the solutions obtained by \textbf{SAV-CN} are different from those solved by the other three methods even under a finer time grid, which demonstrates advantage of exponential integrators and high-order schemes for capturing oscillatory waves. For detailed comparisons, we give the section views of $|u|$ when $x = y$ in the domain $[-15, 15]$ at $t=10$ solved by four methods i.e., \textbf{SAV-CN}, \textbf{SAV-IF}, \textbf{SAV-IFGRK4} and \textbf{SAV-IFGRK6}. The first row of Figure \ref{section view} displays the simulation results under the time step $\tau = 0.05$ while the second row gives the simulation results under the time step $\tau = 0.02$. One can observe that the \textbf{SAV-IFGRK4} and \textbf{SAV-IFGRK6} can capture the oscillatory waves well even under a relatively large time step. Due to the limitation of accuracy, the oscillatory waves failed to be characterized perfectly by \textbf{SAV-IF} under the time step $\tau = 0.05$, but is well captured as we refine the time grid to $\tau = 0.02$. It is evident that the solutions obtained by \textbf{SAV-CN} are different from those solved by the other three methods even under a finer time grid, which demonstrates the advantage of exponential integrators and high-order schemes for capturing oscillatory waves.
\begin{figure}[H]
				\centering
				\includegraphics[width = 0.8\textwidth]{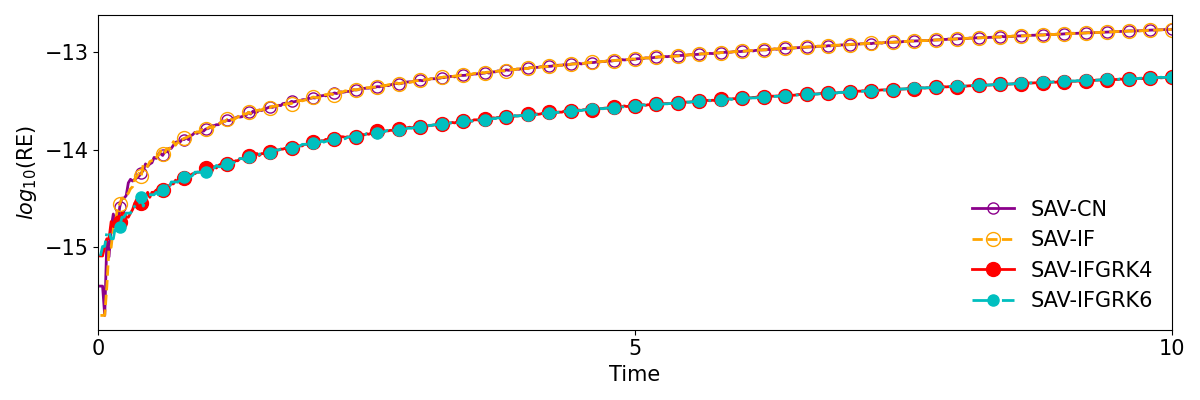}
				\caption{ \label{energy2d} Energy error of the 2D NLSW solved by different methods. }
\end{figure}

Finally, we present the energy evolution of the above methods. We use the following relative energy error 
\begin{equation}
				RE^n = |(E_\mathcal{N}^n - E_\mathcal{N}^0) / E_\mathcal{N}^0|,
\end{equation}
to measure the conservative properties. Figure \ref{energy2d} displays the evolution of the relative energy error solved by four methods with $\tau = 0.05$. We observe that all of them preserve the energy to machine accuracy, which confirms the validity of Theorem \ref{fully-conservation}, \ref{conservation-ifrk}.

\section{Conclusion}
Based on the scalar auxiliary variable approach and the integration factor method, we develop a novel class of linear, energy-preserving exponential integrators for the 2D nonlinear Schr\"odinger equation with wave operator. We rigorously prove that the proposed schemes preserves the discrete energy conservation law. An optimal $H^1$ error estimate is provided for the proposed \textbf{SAV-IF} method for both $\beta > 0$ and $\beta < 0$ without any restrictions on the grid ratios. Specifically, for $\beta > 0$, we prove the convergence result by using the $H^1$ a priori estimate and the equivalence between discrete $L^2$ and $L^p$ norms. For $\beta < 0$, we introduce an improved induction argument to get an unconditional convergence result. Numerical experiments are presented to verify the theoretical results and demonstrate the superiority behavior of our methods over the existing energy-preserving algorithms.

\section{Acknowledgements}
This work is supported by the National Natural Science Foundation of China (12171245, 11971242,
11901513), the Yunnan Fundamental Research Projects (202101AT070208).

\begin{appendices}

\section{Estimates for the local error}\label{A}
				We only provide the estimates of the local error for the $n \geq 1$ and $\gamma = 1$ in Lemma \ref{local-error} due to limitation of space. We begin with providing an estimate of $R_1^{n+1}$. Let $s = \tau$ in the first equation of \eqref{proj-exact} and subtract the resulting equation from the first equation of \eqref{local-error}
				to get 
				\begin{equation*}
				\begin{aligned}
								R_1^{n+1}   &= \tau R\big(t_n + \tfrac{\tau}{2}\big) e^{12}\big(\tfrac{\tau}{2}\big) f^\star(U\big(t_n + \tfrac{\tau}{2}\big)) - \int_0^\tau R(t_n+\sigma) e^{12}(\tau - \sigma) f^\star(U(t_n+\sigma)) d\sigma \\ 
														&\quad +\tau (\tfrac{R(t_n) + R(t_{n+1})}{2} - R(t_n + \tfrac{\tau}{2})) e^{12}\big(\tfrac{\tau}{2}\big) f^\star(U(t_n + \tfrac{\tau}{2})) \\ 
														&\quad +\tau \tfrac{R(t_n) + R(t_{n+1})}{2} e^{12}\big(\tfrac{\tau}{2}\big) \left(f_\mathcal{N}(\widetilde{U}^{\star n+\frac{1}{2}}) - f_\mathcal{N} (U(t_n + \tfrac{\tau}{2}))\right) \\ 
														&\quad +\tau \tfrac{R(t_n) + R(t_{n+1})}{2} e^{12}\big(\tfrac{\tau}{2}\big) \left(f_\mathcal{N} (U\big(t_n + \tfrac{\tau}{2}\big)) - f(U\big(t_n + \tfrac{\tau}{2}\big))\right) \\ 
														&\quad +\tau \tfrac{R(t_n) + R(t_{n+1})}{2} e^{12}\big(\tfrac{\tau}{2}\big) \left(f(U\big(t_n + \tfrac{\tau}{2}\big)) - f^{\star}(U\big(t_n + \tfrac{\tau}{2}\big))\right) \\
								&:= \sum\limits_{j=1}^5 T^u_j
				\end{aligned}
				\end{equation*}
				We expand $T_2^u$ into Taylor series with integral remainder, then use the Cauchy-Schwarz inequality, the definition of $r$ and Lemma \ref{spectral-radius} to get
				\begin{equation*}
				\begin{aligned}
								|T^u_2|_{\mathcal{N},1}^2 &\leq C\tau^5 \int_{t_n}^{t_{n+1}} |\partial_{tt} R(t)|^2 dt \leq C\tau^5 \int_{t_n}^{t_{n+1}} \|\partial_t U(t)\|^2 + \|\partial_{tt} U(t)\|^2 dt \\ 
								&\leq C (\|\partial_t U\|_{L^\infty(0, T; L^2(\Omega))}^2 + \|\partial_{tt} U\|_{L^\infty(0, T; L^2(\Omega))}) \tau^6.
				\end{aligned}
				\end{equation*}
				The mean-value theorem, Lemma \ref{approx-prop-continuous}, \ref{norm_equal} and the Cauchy-Schwarz inequality give us
				\begin{equation*}
				\begin{aligned}
								|T^u_3|_{\mathcal{N},1}^2 &\leq C \tau^2\| f^\prime( \xi \widetilde{U}^{\star n+\frac{1}{2}}+(1-\xi) U\big(t_n + \tfrac{\tau}{2}\big) )\|_{l^\infty}^2 \|\widetilde{U}^{\star n+\frac{1}{2}} - U\big(t_n + \tfrac{\tau}{2}\big)\|_{l^2}^2 \\ 
																		 &\leq C\tau^2 (\| \widetilde{U}^{\star n+\frac{1}{2}} - U^{\star}\big(t_n + \tfrac{\tau}{2}\big)\|^2 + \|I_\mathcal{N} U^{\star}\big(t_n + \tfrac{\tau}{2}\big) - I_\mathcal{N} U\big(t_n + \tfrac{\tau}{2}\big)\|^2) \\ 
																		 &\leq C\tau^2 (\|\widetilde{U}^{n+\frac{1}{2}} - U\big(t_n + \tfrac{\tau}{2}\big)\|^2 + \|U^{\star}\big(t_n + \tfrac{\tau}{2}\big) - U\big(t_n + \tfrac{\tau}{2}\big)\|^2) \\ 
																		 &\leq C (\|\partial_t U\|^2_{L^\infty(0, T, L^2(\Omega))} + \|\partial_{tt} U\|^2_{L^\infty(0, T, L^2(\Omega))} )\tau^6 + C\|U\|^2_{L^\infty(0,T; H_s^m(\Omega))} \tau^2 N^{-2m}.
				\end{aligned}
				\end{equation*}
Analogously, it is easy to get 
\begin{equation*}
				\|T^u_4\|_{l^2}^2 \leq C \|U\|^2_{L^\infty(0,T; H_s^m(\Omega))} \tau^2 N^{-2m}, \quad \|T^u_5\|_{l^2}^2 \leq C \|U\|^2_{L^\infty(0,T; H_s^m(\Omega))} \tau^2 N^{-2m}.
\end{equation*}
Finally, we estimate $T^u_1$ which is caused by numerical quadrature. Let
\begin{equation*}
				\Theta(\sigma) =  e^{12} (\tau - \sigma) f^\star(U(t_n + \sigma)),
\end{equation*}
 then apply the Taylor's formula with integral remainder and the Cauchy-Schwarz inequality to see
\begin{equation}\label{T11-cc}
\begin{aligned}
				|T^u_1|_{\mathcal{N},1}^2 = \left|\int_0^\tau \int_{\frac{\tau}{2}}^\sigma \partial_{tt}\Theta(t) (\sigma - t) dt d\sigma\right|_{\mathcal{N}, 1}^2 \leq C \tau^5 \int_0^\tau  |\partial_{tt}\Theta(t)|_{\mathcal{N}, 1}^2 dt.
\end{aligned}
\end{equation}
A further expansion of $\partial_{tt}\Theta(t)$ leads to
\begin{equation*}
				\partial_{tt}\Theta(t) =  \partial_{tt} e^{12}(\tau - t) f^\star (U(t_n + t))+2 \partial_t e^{12}(\tau - t) \partial_t f^\star (U(t_n + t))+e^{12}(\tau - t) \partial_{tt} f^\star (U(t_n + t)).
\end{equation*}
Since
\begin{equation}\label{T11-auxilary}
\begin{aligned}
				&|\partial_t e^{12}(\tau - t)\partial_t f^\star(U(t_n + t)) |^2_{\mathcal{N},1} = \sum\limits_{p=1}^{N-1}\sum\limits_{q=1}^{N-1} \lambda_{pq}^2 |\partial_t \widehat{e}^{12}_{pq}(\tau - t)|^2 |\partial_t \widehat{f}_{pq}|^2 
				\leq  C\| \partial_t U(t_n + t)\|_{H^1_s}^2, \\ 
				&|\partial_{tt} e^{12}(\tau - t) f^\star (U(t_n + t)) |^2_{\mathcal{N},1} = \sum\limits_{p=1}^{N-1}\sum\limits_{q=1}^{N-1} \lambda_{pq}^2 |\partial_{tt} \widehat{e}^{12}_{pq}(\tau - t)|^2 |\widehat{f}_{pq}|^2 \leq C\|U(t_n + t)\|_{H_s^2}^2, \\ 
				&|e^{12}(\tau - t) \partial_{tt} f^\star (U(t_n + t))|^2_{\mathcal{N},1} \leq C ( \|\partial_{tt} U(t_n+t)\|^2 + \|\partial_t U(t_n+t)\|^2),
\end{aligned}
\end{equation}
we get by combination of \eqref{T11-cc} and \eqref{T11-auxilary} that
\begin{equation*}\label{T11}
\begin{aligned}
				|T^{u}_1|_{\mathcal{N},1}^2 \leq C \tau^6  (\|U\|_{L^\infty(0,T; H_s^2(\Omega))}^2 + \|\partial_t U\|^2_{L^\infty(0,T; H_s^1(\Omega))} + \|\partial_{tt} U\|_{L^\infty(0,T; L^2(\Omega))}^2 ).
\end{aligned}
\end{equation*}
Applications of the triangular equation and the hypotheses of Theorem \ref{main-thm}, \ref{main-thm2} give us
\begin{equation*}
				|R_1^{n+1}|^2_{\mathcal{N},1} \leq  C\tau^2( \tau^4 +  N^{-2m}).
\end{equation*}
The estimates of $R_1^{n+1}$ under the discrete $H^2$ norm and of $R_2^{n+1}$ are analogous, thus we omit them here. 

Finally, we will provide the estimate of $R_3^{n+\frac{1}{2}}$. Subtracting \eqref{exact-half} and the last equation in \eqref{local-error} to get
\begin{equation}\label{xi_r}
\begin{aligned}
				R_3^{n+\frac{1}{2}} &= \delta_t R^{n+\frac{1}{2}} - \partial_t R\big(t_n + \tfrac{\tau}{2}\big) + E_1^{n+\frac{1}{2}}\\
														&\quad - \Re (f_\mathcal{N}(\widetilde{U}^{\star n+\frac{1}{2}}), \mathcal{A}(U^{\star}(t_n), U^{\star}(t_{n+1}), V^{\star}(t_n), V^{\star}(t_{n+1})) )_{l^2} \\ 
												    &\quad + \Re (f^\star(U\big(t_n + \tfrac{\tau}{2}\big)), \mathcal{A}(U^{\star}(t_n), U^{\star}(t_{n+1}), V^{\star}(t_n), V^{\star}(t_{n+1})) + E_2^{n+\frac{1}{2}})_{l^2}.
\end{aligned}
\end{equation}
The RHS of \eqref{xi_r} will be decomposed into the following parts that will be estimated one by one
\begin{equation*}
\begin{aligned}
				\sum\limits_{j=1}^{7} T^r_j &=  \delta_t R^{n+\frac{1}{2}} - \partial_t R\big(t_n + \tfrac{\tau}{2}\big) + E_1^{n+\frac{1}{2}} +  \Re ( f^\star(U\big(t_n + \tfrac{\tau}{2}\big)), E_2^{n+\frac{1}{2}})_{l^2} \\ 
				&+ \Re (f^\star(U\big(t_n + \tfrac{\tau}{2}\big))-f(U\big(t_n + \tfrac{\tau}{2}\big)), \mathcal{A}(U^{\star}(t_n), U^{\star}(t_{n+1}), V^{\star}(t_n), V^{\star}(t_{n+1})))_{l^2}  \\ 
				&+ \Re (f(U\big(t_n + \tfrac{\tau}{2}\big))-f(\widetilde{U}^{n+\frac{1}{2}}), \mathcal{A}(U^{\star}(t_n), U^{\star}(t_{n+1}), V^{\star}(t_n), V^{\star}(t_{n+1})))_{l^2} \\ 
				&+ \Re (f(\widetilde{U}^{n+\frac{1}{2}})-f(\widetilde{U}^{\star n+\frac{1}{2}}), \mathcal{A}(U^{\star}(t_n), U^{\star}(t_{n+1}), V^{\star}(t_n), V^{\star}(t_{n+1})))_{l^2} \\ 
				&+ \Re (f(\widetilde{U}^{\star n+\frac{1}{2}}) - f_\mathcal{N}(\widetilde{U}^{\star n+\frac{1}{2}}), \mathcal{A}(U^{\star}(t_n), U^{\star}(t_{n+1}), V^{\star}(t_n), v^{\star}(_{n+1})))_{l^2}.
\end{aligned}
\end{equation*}
By using the Taylor's formula and some direct calculations, we can analogously derive that
\begin{equation*}\label{T2-other}
\begin{aligned}
				&|T^r_1|^2 \leq C \tau^4, \ |T^r_2|^2  \leq C N^{-2m+2}, \ |T^r_4|^2 \leq C N^{-2m}, \\ 
				&|T^r_5|^2 \leq C \tau^4, \ |T^r_6|^2  \leq C N^{-2m}, \ |T^r_7|^2 \leq C(N^{-2m} + \tau^4).
\end{aligned}
\end{equation*}

We finally estimate $T^r_3$. The definition of $E_2^{n+1/2}$ \eqref{rho-2} gives
\begin{equation}\label{Lambda2}
\begin{aligned}
				E_2^{n+\frac{1}{2}} &= \  \frac{1}{2} \int_0^{\frac{\tau}{2}} \left(e^{22}(-\sigma) - e^{22}\big(\tfrac{\tau}{2}-\sigma\big)\right) g^\star(U\big(t_n + \sigma + \tfrac{\tau}{2}\big)) d\sigma \\ 
										&\quad +  \frac{1}{2} \int_0^{\tfrac{\tau}{2}} e^{22}\big(\tfrac{\tau}{2} - \sigma\big) \left(g^\star(U\big(t_n + \sigma + \tfrac{\tau}{2}\big)) - g^\star(U(t_n + \sigma))\right) d\sigma \\
										&= E_{21}^{n+\frac{1}{2}} + E_{22}^{n+\frac{1}{2}}.
\end{aligned}
\end{equation}
According to Lemma \ref{approx-prop-continuous}, \ref{spectral-radius} and Taylor's formula, we  have $\|E_{22}^{n+\frac{1}{2}}\|_{l^2}^2 \leq C \tau^4$. For $E_{21}^{n+\frac{1}{2}}$, we expand it into discrete sine series, then use the Cauchy-Schwarz inequality and Taylor's formula with integral remainder to get
\begin{equation*}\label{Lambda21}
\begin{aligned}
				\|E_{21}^{n+\frac{1}{2}}\|_{l^2}^2 &\leq C\tau  \int_0^{\frac{\tau}{2}} \| \big(e^{22}(-\sigma) - e^{22} \big(\tfrac{\tau}{2} - \sigma\big)\big) g^\star(U\big(t_n + \sigma + \tfrac{\tau}{2}\big)) \|_{l^2}^2 d\sigma \\ 
																&\leq C\tau \sum\limits_{p=1}^{N-1} \sum\limits_{q=1}^{N-1} \int_0^{\frac{\tau}{2}} |\widehat{e}_{pq}^{22}(-\sigma) - \widehat{e}_{pq}^{22}(\tfrac{\tau}{2}-\sigma)|^2 |\widehat{g}_{pq}|^2 d\sigma \\ 
																&\leq C\tau^2 \sum\limits_{p=1}^{N-1}\sum\limits_{q=1}^{N-1} \int_0^{\frac{\tau}{2}} \int_0^{\frac{\tau}{2}} |\partial_t \widehat{e}_{pq}^{22}(s-\sigma)|^2|\widehat{g}_{pq}|^2 dt d\sigma \\ 
																& \leq C\tau^3 \int_{t_{n}+\frac{\tau}{2}}^{t_{n+1}} \|g\|^2_{H_s^1} dt \leq C \|U\|_{L^\infty(0, T; H_s^1(\Omega))}  \tau^4,
\end{aligned}
\end{equation*}
where the estimate $|\partial_t e_{pq}^{22}|^2 \leq C \lambda_{pq}^2$ is employed here. The direct combination of the above estimates and the triangular inequality gives us the error estimate of $R_3^{n+\frac{1}{2}}$.
\section{Proof of the Lemma \ref{sigma-h1}}\label{C}
For the sake of simplicity, we let 
\begin{equation*}
				\{\nabla_h u\}_{jk} = (\delta^+_x u_{jk}, \delta^+_y u_{jk}). 
\end{equation*}
\begin{lem}\label{B1}
				For any $u, v \in \mathcal{M}_\mathcal{N}$, we have $\|\nabla_h (u v) \|_{l^\infty} \leq 2( \|v\|_{l^\infty} |u|_{h,1} + \|u\|_{l^\infty} |v|_{h,1})$.
\end{lem}
\begin{proof}
				We first recall the following equalities 
				\begin{equation*}
				\begin{aligned}
								\delta_x^+ (u_{jk} v_{jk}) = \delta_x^+ u_{jk} \cdot v_{jk} + u_{j+1k} \cdot \delta_x^+ v_{jk} \ \text{and} \	\delta_y^+ (u_{jk} v_{jk}) &= \delta_y^+ u_{jk} \cdot v_{jk} + u_{jk+1} \cdot \delta_y^+ v_{jk}. 
				\end{aligned}
				\end{equation*}
				Then, we use the inequality $(a + b)^2 \leq 2(a^2 + b^2)$ to get
				\begin{equation*}
				\begin{aligned}
								\|\delta_x^+ (uv)\|^2_{l^2} &\leq 2\sum\limits_{j=0}^{N-1}\sum\limits_{k=1}^{N-1} |\delta_x^+ u_{jk}|^2 |v_{jk}|^2 + 2\sum\limits_{j=0}^{N-1}\sum\limits_{k=1}^{N-1} |u_{j+1k}|^2 |\delta_x^+ v_{jk}|^2 \\ 
																				&\leq 2\|v\|^2_{l^\infty} \sum\limits_{j=0}^{N-1}\sum\limits_{k=1}^{N-1} |\delta_x^+ u_{jk}|^2 + 2 \|u\|^2_{l^\infty} \sum\limits_{j=1}^N \sum\limits_{k=1}^{N-1} |\delta_x^+ v_{jk}|^2. \\ 
				\end{aligned}
				\end{equation*}	
				Analogously, 
				\begin{equation*}
				\begin{aligned}
								\|\delta_y^+ (uv)\|^2_{l^2} &\leq 2 \|v\|_{l^\infty}^2 \sum\limits_{j=1}^{N-1}\sum\limits_{k=0}^{N-1} |\delta_x^+ u_{jk}|^2  + 2 \|u\|_{l^\infty}^2 \sum\limits_{j=1}^{N-1} \sum\limits_{k=1}^{N} |\delta_x^+ v_{jk}|^2.
				\end{aligned}
				\end{equation*}
				The definition of operator $\nabla_h$ and the above inequalities gives us
				\begin{equation*}
								\|\nabla_h (uv)\|^2_{l^2} \leq  2( \|v\|^2_{l^\infty} |u|^2_{h,1} + \|u\|^2_{l^\infty} |v|^2_{h,1}),
				\end{equation*}
				and the proof of Lemma \ref{B1} is thus completed.
\end{proof}
Now, we prove Lemma \ref{sigma-h1}. Adding and subtracting some intermediate terms in $E_2^{l+\frac{1}{2}}$, we get
\begin{equation*}
\begin{aligned}
				E_2^{l+\frac{1}{2}} &= \zeta^{l+\frac{1}{2}}\widetilde{f}_\mathcal{N}^{\star l+\frac{1}{2}} + r^{l+\frac{1}{2}}(\widetilde{f}_\mathcal{N}^{\star l+\frac{1}{2}}-\widetilde{f}_\mathcal{N}^{l+\frac{1}{2}}) .
\end{aligned}
\end{equation*}
In view of the Lemma \ref{norm_equal}, the hypothesis of Theorem \ref{main-thm2}, and the boundedness of $r^{l+1}$ in \eqref{bound-R}, it holds that
\begin{equation}\label{sigma-final}
\begin{aligned}
				|E_2^{l+\frac{1}{2}}|_{\mathcal{N},1} &\leq C |\zeta^{l+\frac{1}{2}}| \|\nabla_h \widetilde{f}_\mathcal{N}^{\star l+\frac{1}{2}}\|_{l^2} + C|r^{l+\frac{1}{2}}| \| \nabla_h (\widetilde{f}_\mathcal{N}^{\star l+\frac{1}{2}}-\widetilde{f}_\mathcal{N}^{l+\frac{1}{2}}) \|_{l^2} \\ 
																 &\leq C (|\zeta^{l+\frac{1}{2}} | + \|\nabla_h( \widetilde{g}^{\star l+\frac{1}{2}} - \widetilde{g}^{l+\frac{1}{2}} )\|_{l^2} + \|\nabla_h \widetilde{g}^{l+\frac{1}{2}}\|_{l^2} \| \widetilde{g}^{\star l+\frac{1}{2}} - \widetilde{g}^{l+\frac{1}{2}} \|_{l^2}).
\end{aligned}
\end{equation}
According to the boundedness of $u^{l-1}$ and $u^l$ from the induction, Lemma \ref{norm-equivalence2d}, \ref{B1}, we arrive at
\begin{equation*}
\begin{aligned}
				\|\nabla_h \widetilde{g}^{l+\frac{1}{2}}\|_{l^2} &\leq C  \|\widetilde{u}^{l+\frac{1}{2}}\|_{l^\infty}^2 |\widetilde{u}^{l+\frac{1}{2}}|^2_{\mathcal{N},1} \leq C,
\end{aligned}
\end{equation*}
and 
\begin{equation*}
\begin{aligned}
				\| \widetilde{g}^{\star l+\frac{1}{2}} - \widetilde{g}^{l+\frac{1}{2}} \|_{l^2} &\leq C \|\widetilde{\xi}^{l+\frac{1}{2}}\|_{l^2}  (\|\widetilde{U}^{\star,l+\frac{1}{2}}\|^2_{l^\infty}+ \|\widetilde{U}^{\star,l+\frac{1}{2}}\widetilde{u}^{l+\frac{1}{2}}\|_{l^\infty} + \|\widetilde{u}^{l+\frac{1}{2}}\|_{l^\infty}^2) \\ 
																																							&\leq C \|\widetilde{\xi}^{l+\frac{1}{2}}\|_{l^2} \leq C (\|\xi^{l-1}\|_{l^2} + \|\xi^l\|_{l^2}).
\end{aligned}
\end{equation*}
Analogously, we derive the estimate for $E_{21}^{l+\frac{1}{2}} = \|\nabla_h ( \widetilde{g}^{\star l+\frac{1}{2}} - \widetilde{g}^{l+\frac{1}{2}} )\|_{l^2}$ as follows:
\begin{equation*}
\begin{aligned}
				E_{21}^{l+\frac{1}{2}} &\leq C( |\widetilde{U}^{\star l+\frac{1}{2}}|_{h,1}\|\widetilde{u}^{l+\frac{1}{2}}\|_{l^\infty} \|\widetilde{\xi}^{l+\frac{1}{2}}\|_{l^\infty} + \|\widetilde{U}^{\star l+\frac{1}{2}}\|_{l^\infty} |\widetilde{u}^{l+\frac{1}{2}}|_{h,1} \|\widetilde{\xi}^{l+\frac{1}{2}}\|_{l^\infty}  ) \\ 
														 &+ C( |\widetilde{U}^{\star l+\frac{1}{2}}|_{h,1}\|\widetilde{u}^{l+\frac{1}{2}}\|_{l^\infty} \|\widetilde{\xi}^{l+\frac{1}{2}}\|_{l^\infty} + \|\widetilde{U}^{\star l+\frac{1}{2}}\|_{l^\infty} |\widetilde{u}^{l+\frac{1}{2}}|_{h,1} \|\widetilde{\xi}^{l+\frac{1}{2}}\|_{l^\infty}  ) \\ 
														 &+ C( |\widetilde{u}^{l+\frac{1}{2}}|_{h,1}\|\widetilde{u}^{l+\frac{1}{2}}\|_{l^\infty}\|\widetilde{\xi}^{l+\frac{1}{2}}\|_{l^\infty} + \|\widetilde{u}^{l+\frac{1}{2}}\|_{l^\infty} |\widetilde{u}^{l+\frac{1}{2}}|_{h,1} \|\widetilde{\xi}^{l+\frac{1}{2}}\|_{l^\infty}  )  \\ 
														 &+ C( \|\widetilde{U}^{\star l+\frac{1}{2}}\|_{l^\infty}^2 |\widetilde{\xi}^{l+\frac{1}{2}}|_{h,1} + \|\widetilde{U}^{\star l+\frac{1}{2}}\|_{l^\infty}\| \widetilde{u}^{l+\frac{1}{2}}\|_{l^\infty} |\widehat{\xi}^{l+\frac{1}{2}}|_{h,1} + \|\widetilde{u}^{l+\frac{1}{2}}\|_{l^\infty}^2 |\widetilde{\xi}^{l+\frac{1}{2}}|_{h,1} ) \\
														 &\leq C (|\xi^{l-1}|_{\mathcal{N},2} + |\xi^l|_{\mathcal{N},2})
\end{aligned}
\end{equation*}
Notice that the norm equivalence Lemma \ref{norm-equivalence2d}, the discrete Sobolev inequality Lemma \ref{Sobolev_embedding2d_Linfty} and
\begin{equation*}
\begin{aligned}
				\|\xi^l\|_{l^2} &= \|I_\mathcal{N} \xi^l\| \leq \|I_\mathcal{N} \xi^l\|_{H_s^2} = |\xi^l|_{\mathcal{N},2}, \\ 
				|\xi^l|_{\mathcal{N},1} &= \|I_\mathcal{N} \xi^l\|_{H_s^1} \leq C \|I_\mathcal{N} \xi^l\|_{H_s^2} = C |\xi^l|_{\mathcal{N},2}, \\ 
\end{aligned}
\end{equation*}
are used. Inserting the above estimates into \eqref{sigma-final} leads to the final result.
\end{appendices}

\bibliographystyle{abbrv}
\bibliography{Reference.bib}

\end{document}